\documentclass{svmult}
\usepackage{amsmath}
\usepackage{amsfonts}
\usepackage{amstext}  
\usepackage{amssymb}  
\usepackage{esint}    
\usepackage{cite}
\usepackage{mathptmx}
\usepackage{helvet}
\usepackage{courier}
\usepackage{makeidx}
\usepackage{graphicx}
\usepackage{multicol}
\usepackage{footmisc}

\providecommand\C{}
\providecommand\F{}

\providecommand\P{}
\providecommand\Q{}
\providecommand\Z{}

\renewcommand{\C}{\mathbb{C}}
\renewcommand{\F}{\mathbb{F}}

\renewcommand{\P}{\mathbb{P}}
\renewcommand{\Q}{\mathbb{Q}}
\renewcommand{\Z}{\mathbb{Z}}

\newcommand{\cF}{\mathcal{F}}
\newcommand{\cH}{\mathcal{H}}
\newcommand{\cJ}{\mathcal{J}}

\newcommand{\cL}{\mathcal{L}}
\newcommand{\cM}{\mathcal{M}}
\newcommand{\cO}{\mathcal{O}}
\newcommand{\cR}{\mathcal{R}}
\newcommand{\cW}{\mathcal{W}}
\newcommand{\cX}{\mathcal{X}}

\providecommand\i{}
\renewcommand{\i}{\imag}
\newcommand{\modp}{\mod p}
\newcommand{\Fcross}{\F_{p}^{*}}

\newcommand{\knu}{k_{r}}
\newcommand{\conj}[1]{\overline{#1}}
\newcommand{\knus}{k_{r}^{*}}

\newcommand{\Res}{\textnormal{Res}}
\newcommand{\Teich}{\textnormal{Teich}}
\newcommand{\Frob}{\textnormal{Frob}}
\newcommand{\mat}[2][ccccccccccccccccccccccccc]{
  \left(
    \begin{array}{#1}
      #2\\
    \end{array}
  \right)
}

\numberwithin{equation}{chapter}

\begin{document}

\title*{Introduction to Arithmetic Mirror Symmetry}
\author{Andrija Peruni\v{c}i\'c}
\institute{Andrija Peruni\v{c}i\'c \at Queen's University\\
Kingston, ON K7L 3N6, Canada\\
\email{perunicic@mast.queensu.ca}}
\maketitle

\abstract{We describe how to find period integrals and Picard-Fuchs
  differential equations for certain one-parameter families of
  Calabi-Yau manifolds. These families can be seen as varieties over a
  finite field, in which case we show in an explicit example that the
  number of points of a generic element can be given in terms of
  $p$-adic period integrals. We also discuss several approaches to
  finding zeta functions of mirror manifolds and their
  factorizations. These notes are based on lectures given at the Fields
  Institute during the thematic program on
  \emph{Calabi-Yau Varieties: Arithmetic, Geometry, and Physics.}}

\section{Introduction}

The mirror conjecture is an important early result \cite{LLY} in
mirror symmetry which suggests that counting rational curves on a
Calabi-Yau threefold, an enumerative problem, can be done in terms of
Hodge theory and period integrals on its mirror partner. An arithmetic
counterpart to these ideas that was introduced by Candelas, de la
Ossa, and Rodriguez-Villegas in \cite{CORV1} can be stated as
follows. Let $\cM$ denote the one-parameter family of quintic
threefolds, consisting of hypersurfaces
\begin{equation}
X_{\psi}\colon\left\{ 0=x_{1}^{5}+x_{2}^{5}+x_{3}^{5}+x_{4}^{5}+x_{5}^{5}-5\psi x_{1}x_{2}x_{3}x_{4}x_{5}\right\} \subseteq\P^{4},\label{eq:quintic-def}
\end{equation}
where we exclude those $\psi$ which give singular fibers. This family
is typically defined over $\C$, but if we take $\psi$ to be an element
of a finite field $k$, then we can consider $X_{\psi}$ as a variety
defined over $k$. It then turns out that the number of points of
$X_{\psi}$ can be given in terms of a $p$-adic version of certain
periods on $X_{\psi}$. The purpose of these notes is to describe
in detail the techniques and ideas behind this calculation and some
related work arising from the intersection of arithmetic and mirror
symmetry. 

A more detailed overview is in order. A period of $X_{\psi}$ (defined
over $\C$) is an integral $\int_{\gamma}\omega$ of the unique
holomorphic top-form $\omega$ on $X_{\psi}$ over some $3$-cycle
$\gamma$.  There are $204$ independent periods of $X_{\psi}$ in total,
owing to the fact that $\dim H_{3}(X_{\psi};\C)=204$. Four of these
$204$ period integrals can also be seen as periods of the
\emph{mirror} family $\cW$. These four period integrals are functions
of the parameter $\psi$, and are, in fact, all of the solutions of
an\emph{ }ordinary differential equation $\cL f(\psi)=0$ called the
\emph{Picard-Fuchs} equation of $\cW$, where for
$\lambda=1/(5\psi)^{5}$ and $\vartheta=\lambda\frac{\D}{\D\lambda}$ we
define
\begin{equation}
\cL:=\vartheta^{4}-5\lambda\prod_{i=1}^{4}(5\vartheta+i).\label{eq:quintic-PFE}
\end{equation}
This is a hypergeometric differential equation with fundamental solution
around $\lambda=0$ given by 
\begin{equation}
\label{eq:fundamental-period-quintic}
\varpi_{0} =\sum_{m=0}^{\infty}\frac{(5m)!}{(m!)^{5}}\lambda^{m} = \sum_{m=0}^{\infty}\frac{\Gamma(5k+1)}{\Gamma(k+1)^{5}}\lambda^{m}.
\end{equation}
Consider now each $X_{\psi}$ as a variety over the finite field $k=\F_{p}$
with $p$ elements and assume that $5\nmid(p-1)$.
The number of points $N(X_\psi)$ on $X_{\psi}$ with coordinates
in $k$ is given by the expression 
\begin{equation}
N(X_\psi)=1+p^{4}+\sum_{m=1}^{p-2}\frac{G_{5m}}{G_{m}^{5}}\Teich^{m}(\lambda),\label{eq:quintic-period-nr-points}
\end{equation}
where $\Teich(\lambda)$ is the Teichm\"uller lifting of $\lambda$
to the $p$-adic numbers $\Z_{p}$, and $G_{m}$ is a \emph{Gauss sum
}proportional to the $p$-adic gamma function. This expression for
$N(X_\psi)$ can be seen as a $p$-adic analog of the hypergeometric
series (\ref{eq:fundamental-period-quintic}). A way to illustrate
this point is to reduce (\ref{eq:quintic-period-nr-points}) modulo $p$,
\[
N(X_\psi)\equiv\sum_{m=0}^{\lfloor p/5\rfloor}\frac{(5m)!}{(m!)^{5}}\lambda^{m}\mod p,
\]
which is a truncation of (\ref{eq:fundamental-period-quintic}). In
fact, the number of $\F_{p}$-rational points on $X_{\psi}$ can be
written as a modulo $p^{5}$ expression (\ref{eq:nr-points-as-all-solns-and-semiperiod})
involving all of the solutions of (\ref{eq:quintic-PFE}), as well as 
an additional term arising from a so-called \emph{semi-period}.

Arithmetic of varieties appearing in the context of mirror symmetry
can also be studied through their zeta functions, defined for a
variety $X$ over a finite field $\F_q$ with $q=p^n$ elements by
\[
Z(X,T) = \exp \left( \sum_{r = 1}^\infty N_r(X) \frac{T^r}{r} \right),
\]
where $N_r(X)$ denotes the number of points of
$X \otimes_{\F_q} \bar \F_q$ rational over $\F_{q^r}$. It turns out
that the zeta function of $X_{\psi}$ contains all the terms appearing
in the zeta function of its mirror manifold, and the terms not
appearing in the mirror zeta function exhibit interesting factorization
properties. In the context of mirror symmetry, zeta functions were
first considered by Candelas, de la Ossa, and Rodriguez-Villegas in
\cite{CORV2}. Due to the explicit nature of their calculation and a
large overlap with point counting in terms of period integrals we will
focus on their exposition. However, we will also discuss other
approaches to calculating zeta functions and their factorizations that
are of a more conceptual nature.

The notes are organized as follows. In
Section~\ref{sec:Period-Integrals} we more carefully define period
integrals. In Section~\ref{sec:Differentials-on-Hypersurfaces} we
discuss differentials on hypersurfaces and relations between them.
These relations enable us to find Picard-Fuchs equations satisfied by
the periods, and by solving them the periods themselves, in
Sections~\ref{sec:Determining-Picard-Fuchs-Eq} and
\ref{sec:Finding-the-Periods}.  We then switch gears and talk about
counting points and what we mean by $p$-adic periods in
Section~\ref{sec:Character-Formulas}. Finally, we discuss zeta
functions of mirror manifolds and their factorizations in
Section~\ref{sec:Zeta-Functions-and}.

\section{Periods and Picard-Fuchs Equations}

\subsection{Period Integrals\label{sec:Period-Integrals}}

Let $\pi\colon\cX\to B$ be a proper submersion defining a \emph{family}
of smooth $n$-dimensional K\"ahler manifolds. Ehresmann's
Fibration Theorem \cite[Theorem 9.1]{VOIS1} then implies that for each $\psi\in B$ there exists
an open set $U$ containing $\psi$, and a diffeomorphism $\varphi$
such that the diagram 
\[
\begin{array}{ccc}
\pi^{-1}(U) & \underrightarrow{\varphi} & U\times X_{\psi}\\
\searrow & \circlearrowleft & \swarrow\\
 & U
\end{array}
\]
commutes. In other words, $\pi$ is a locally trivial fibration. In
these notes, the fiber $X_{\psi}:=\pi^{-1}(\psi)$ is a nonsingular
projective hypersurface for each $\psi\in B$. Let $\cF$ be a sheaf
on $\cX$. Mapping $\cF$ to the direct image sheaf $\pi_{*}\cF$
on $B$, determined by $\pi_{*}\cF(U):=\cF(\pi^{-1}(U))$, defines
a covariant functor from sheaves on $\cX$ into sheaves on $B$. This
functor is left exact, but in general not right exact. In fact, $R^{k}\pi_{*}\cF$
is the sheafification of the presheaf $H^{k}(\pi^{-1}(-),\cF\mid_{(-)})$.
Consider the case $k=n$ and $\cF=\underbar{\ensuremath{\C}}$, the
constant sheaf valued in $\C$. Stalks are determined on contractible
open sets, so for $U\ni\psi$ such that $\pi^{-1}(U)\cong U\times X_{\psi}$
we have 
\[
(R^{n}\pi_{*}\underbar{\ensuremath{\C}})_{\psi}\cong H^{n}(X_{\psi},\C).
\]
The groups on the right are canonically isomorphic for all $\psi\in U$,
which means that $(R^{n}\pi_{*}\underbar{\ensuremath{\C}})\mid_{U}$ defines a locally constant
sheaf on $B$, i.e., a local system $H$ of complex vector spaces.
Tensoring with the structure sheaf $\cO_{B}$, we obtain a locally
free $\cO_{B}$-module $\cH=H\otimes\cO_{B}$ which canonically admits
the \emph{Gauss-Manin connection}
$$
\nabla\colon\cH\to \cH\otimes\Omega_{B}^{1}
$$
defined by
$$
\nabla\left(\sum_{i}\alpha_{i}\sigma_{i}\right):= \sum_{i}\sigma_{i}\otimes d\alpha_{i},
$$
where $\{\sigma_{i}\}$ is any local basis of $H$, $\Omega_{B}^{1}$
is the sheaf of holomorphic $1$-forms on $B$, and $\alpha_{i}\in\cO_{B}$.
This connection can be extended to a map
$\nabla\colon \cH \otimes\Omega_{B}^{k}\to\cH\otimes\Omega_{B}^{k+1}$
by defining $\nabla(\sigma\otimes\omega)=(\nabla\sigma)\wedge\omega$. More details are available in \cite{VOIS1}, for instance.

For any $\psi\in U$, choose a basis of $n$-cycles $\{\gamma_{i}\}$
on $X_{\psi}$ such that the corresponding homology classes generate
$H_{n}(X_{\psi};\mathbb{C})$. We can choose this basis to be dual
to $\{\sigma_{i}\}$ and extend it to nearby fibers. For $s(\psi)\in\Gamma(U,\cH)$
varying holomorphically and $\gamma$ a homology class, we obtain
a holomorphic function $\langle s(-),\gamma\rangle\colon U\to\C$
via the Poincar\'e pairing, 
\[
\langle s(\psi),\gamma\rangle=\int_{\gamma}s(\psi).
\]
The sheaf generated by such functions is called the \emph{period sheaf.
}By the de Rham theorem \cite[Section 4.3.2]{VOIS1} we can think of $s(\psi)$ as a holomorphic
family of differential forms.
\begin{definition}
Let $\omega$ be a holomorphic $n$-form on an $n$-dimensional complex
manifold $X$. Integrals of the form 
\[
\int_{\gamma}\omega\quad\text{for }\gamma\in H_{n}(X;\C)
\]
are called \emph{period integrals (periods)} of $X$ with respect
to $\omega$.
\end{definition}
Extending the (co)homology basis to all of $B$ can lead to nontrivial
monodromy on the fibers of $\cX$, which will in turn induce monodromy
on the periods. However, in these notes we are only interested in the periods locally.
In the case that the parameter space $B$ is one dimensional, the
Gauss-Manin connection is locally given by differentiation $\nabla_{\psi}:=\frac{\D}{\D\psi}$
with respect to the parameter $\psi\in B$, and satisfies

\[
\frac{\D}{\D\psi}\int_{\gamma_{i}}\omega(\psi)=\int_{\gamma_{i}}\frac{\D}{\D\psi}\omega(\psi).
\]
From here on, we are working only with one-parameter families. 
\begin{proposition}
The periods with respect to $\omega(\psi)$ satisfy an ordinary differential
equation of the form 
\[
\frac{\D^{s}f}{\D\psi^{s}}+\sum_{j=0}^{s-1}C_{j}(\psi)\frac{\D^{j}f}{\D\psi^{j}}=0,
\]
where $s$ is a natural number. This equation is called the Picard-Fuchs equation for $\omega(\psi)$.
\end{proposition}
\begin{proof}[given in \cite{MOR}]
\smartqed
Let $\psi$ be the parameter on an open set $U\subseteq B$,
and for $j\in\mathbb{Z}$ define 
\[
v_{j}(\psi)=\frac{\D^{j}}{\D\psi^{j}}\left(\begin{array}{c}
\int_{\gamma_{1}}\omega(\psi)\\
\vdots\\
\int_{\gamma_{r}}\omega(\psi)
\end{array}\right)\in\mathbb{C}(\psi)^{r}.
\]
For $i\in\mathbb{N}^{+}$ and nearby values of $\psi$, the vector
spaces 
\[
V_{i}(\psi):=\text{span}\left\{ v_{0}(\psi),\ldots,v_{i}(\psi)\right\} 
\]
vary together smoothly with respect to $\psi$. Since for a particular
value of $\psi$ each $V_{i}(\psi)\in\mathbb{C}^{r}$, we also have
that $\text{dim}V_{i}(\psi)\leq r.$ Therefore, there is a smallest
$s\leq r$ such that $v_{s}(\psi)\in\text{span\ensuremath{\left\{ v_{0}(\psi),\ldots,v_{s-1}(\psi)\right\} }}$,
giving the equation 
\[
v_{s}(\psi)=-\sum_{i=0}^{s-1}C_{j}(\psi)v_{j}(\psi)
\]
satisfied by $\int_{\gamma_{i}}\omega(\psi)$ for each $\gamma_{i}$,
as claimed.
\qed
\end{proof}
Picard-Fuchs equations can in general have non-period solutions, 
but we will not encounter them in these notes.

\subsection{Differentials on Hypersurfaces\label{sec:Differentials-on-Hypersurfaces}}

Let $\pi\colon\cX\to B$ be a one-parameter family of hypersurfaces
$\{X_{\psi}\}_{\psi\in B}$ in projective space, and let $\omega=\omega(\psi)$
be a top-form on the family. A basic strategy for finding Picard-Fuchs
equations is to express the forms $\frac{\D^{s}\omega}{\D\psi^{s}}$
in terms of a particular basis of forms on $X_{\psi}$, and exploit this description
to find relations between them.
We will now show how to find a basis of forms on $X_{\psi}$
in the first place, by relating them via \emph{residue maps} to 
\emph{rational} forms on projective space which we can write down explicitly.

\subsubsection{The Adjunction Formula and Poincar\'e Residues}

We begin by defining the residue map for differentials with a
simple pole in projective space. 
Throughout this section, $Y$ denotes
an $n$-dimensional compact complex manifold and $X$ a hypersurface
on $Y$. An example to keep in mind is $Y=\P^{4}$
 and $X\subset\P^{4}$ a generic element of  (\ref{eq:quintic-def}). 
A reference for this section is \cite{GRHA}. Recall that the \emph{normal bundle} on $X$ is given by the quotient
$N_{X}=T_{Y}\mid_{X}/T_{X}$ of tangent bundles, that its dual $N_{X}^{*}$
is called the \emph{conormal bundle}, and that the \emph{canonical bundle}
of (any manifold) $X$ is defined as 
\[
K_{X}:=\bigwedge^{n}\Omega_{X}^{1}=\Omega_{X}^{n},
\]
where $\Omega_{X}^{1}=T_{X}^{*}$ . The sections of the
canonical bundle are given by holomorphic
$n$-forms on $X$, which locally look like $\omega=f(z)\D z$, where
$z$ is a local coordinate and $f(z)$ is holomorphic. In the discussion
that follows, $[X]$ is the line bundle associated with the divisor
$X$. 
\begin{proposition}[The Adjunction Formula]
\label{prop:adjunction-formula} For $Y$
and $X$ defined as above we have the isomorphism 
\[
K_{X}\cong \left. K_{Y}\right|_{X}\otimes N_{X}\,.
\]
\end{proposition}
\begin{proof}
\smartqed
From the conormal exact sequence for $X$, 
\[
0\to N_{X}^{*}\to\Omega_{Y}^{1}\mid_{X}\to\Omega_{X}^{1}\to0,
\]
we have that 
\[
K_{Y}|_{X}\cong\bigwedge^{n}\left(\Omega_{Y}^{1}|_{X}\right) \cong N_{X}^{*}\otimes\bigwedge^{n-1}\Omega_{X}^{1} \cong N_{X}^{*}\otimes K_{X}.
\]
Tensoring with $N_{X}$ gives the result.
\qed
\end{proof}
The map on sections corresponding to the adjunction formula is called
the Poincar\'e residue map. To describe it, we need to set up some notation
and observe a few basic facts. Denote by $\Omega_{Y}^{n}(X)$ the
sheaf of meromorphic differentials on $Y$ with a pole of order one
along $X$. Tensoring by a section of $[X]$ provides the isomorphism
\[
\Omega_{Y}^{n}(X)\cong\Omega_{Y}^{n}\otimes[X],
\]
where we are abusing notation and writing $\Omega_{Y}^{n}\otimes[X]$
for $\cO(\Omega_{Y}^{n}\otimes[X])$. The line bundle $[X]$ is given
by transition functions $g_{ij}=f_{i}/f_{j}$, where $f_{i}$ and
$f_{j}$ are local functions of $X$ on open sets $U_{i}$ and $U_{j}$
with nontrivial intersection. Using the product rule then shows that
a section $\D f_{i}$ of the conormal bundle on $X$ can be written
as $g_{ij}\, \D f_{j}$, which means that $[X]\otimes N_{X}^{*}$ has
a nonzero global section $\{\D f_{i}\}$ and is consequently trivial.
Dualizing, we obtain $N_{X}=[X]\mid_{X}$. The above isomorphism and
Proposition \ref{prop:adjunction-formula} then imply that sections
of $\Omega_{Y}^{n}(X)$ correspond to sections of $\Omega_{X}^{n-1}$.
The former are locally given by meromorphic $n$-forms with a single
pole along $X$ and holomorphic elsewhere, 
\[
\omega=\frac{g(z)}{f(z)}\D z_{1}\wedge\ldots\wedge \D z_{n},
\]
where $z=(z_{1},z_{2},\ldots z_{n})$ are the local coordinates on
$Y$, and $X$ is locally given by $f(z).$ If we write $\D f=\sum_{i=1}^{n}\frac{\partial f}{\partial z_{i}}\D z_{i}$,
it follows that for any $i$ such that $\frac{\partial f}{\partial z_{i}}\neq0$,
the form $\omega'$ on $X$ defined by
\begin{equation}
\omega'=(-1)^{i}\frac{g(z)\, \D z_{1}\wedge\ldots\wedge\widehat{\D z_{i}}\wedge\ldots\wedge \D z_{n}}{\partial f/\partial z_{i}}\label{eq:PoincareResidueFormula}
\end{equation}
satisfies 
\[
\omega=\frac{\D f}{f}\wedge\omega'.
\]
The \emph{Poincar\'e residue map} $\text{Res}\colon\Omega_{Y}^{n}(X)\to\Omega_{X}^{n-1}$
can then locally be given by $\omega\mapsto\omega'\mid_{f=0}$. 
\begin{example}
\label{expl:FermatResidueCalculation} Let $\P^2$ have coordinates $[x_1 : x_2 : x_3 ]$. The \emph{Fermat family of elliptic curves} is the one-parameter family of hypersurfaces $Z_{\psi} \subset \P^2$ given by
\[
Z_{\psi}\colon\{F_{\psi}:=x_{1}^{3}+x_{2}^{3}+x_{3}^{3}-3\psi x_{1}x_{2}x_{3}=0\}
\]
for $\psi\in\C\setminus\{\xi_{1},\xi_{2},\xi_{3}\}$, where we exclude
$\xi_{n}=\left(\eul^{2\pi\i/3}\right)^{n}$ as values of the parameter
$\psi$ since they yield singular fibers. Let 
\[
\omega=\frac{1}{F_{\psi}}\left(x_{1}\D x_{2}\wedge \D x_{3}-x_{2}\D
  x_{1}\wedge \D x_{3}+x_{3}\D x_{1}\wedge \D x_{2}\right)
\]
be a section of $\Omega_{\P^{2}}^{2}(Z_{\psi})$. We will directly
compute $\omega'$ on $U_{3}=\{[x_{1}\colon x_{2}\colon x_{3}]\mid x_{3}\neq0\}$
with coordinates $z_{1}=x_{1}/x_{3}$ and $z_{2}=x_{2}/x_{3}$. Since

\[
\D z_{1}=\frac{\partial z_{1}}{\partial x_{1}}\D x_{1}+\frac{\partial z_{2}}{\partial x_{2}}\D x_{2}=\frac{x_{3}\D x_{1}-x_{1}\D x_{3}}{x_{3}^{2}}\quad\text{and}\quad \D z_{2}=\frac{x_{3}\D x_{2}-x_{2}\D x_{3}}{x_{3}^{2}},
\]
we have 
\[
\D z_{1}\wedge \D z_{2}=\frac{1}{x_{3}^{3}}\left(x_{1}\D x_{2}\wedge \D x_{3}-x_{2}\D x_{1}\wedge \D x_{3}+x_{3}\D x_{1}\wedge \D x_{2}\right).
\]
If we let $f:=F_{\psi}\mid_{U_{3}}$, then it follows that 
\begin{align*}
\omega=\frac{x_{3}^{3}}{F_{\psi}}\ \D z_{1}\wedge \D z_{2}= & \frac{x_{3}^{3}}{x_{3}^{3}(\frac{x_{2}^{3}}{x_{3}^{3}}+\frac{x_{2}^{3}}{x_{3}^{3}}+1-3\psi\frac{x_{1}x_{2}x_{3}}{x_{3}^{3}})}\ \D z_{1}\wedge \D z_{2}\\
= & \frac{1}{f}\ \D z_{1}\wedge \D z_{2}.
\end{align*}
We wish to solve for $A(z_{1},z_{2})$ and $B(z_{1},z_{2})$ in $\omega'=A\, \D z_{1}+B\, \D z_{2}$
satisfying 
\[
\omega'\wedge\frac{\D f}{f}=\frac{1}{f}\D z_{1}\wedge \D z_{2}.
\]
Taking $B=0$ and evaluating

\[
A\, \D z_{1}\wedge\frac{1}{f}\left(\frac{\partial f}{\partial z_{1}}\, \D z_{1}+\frac{\partial f}{\partial z_{2}}\, \D z_{2}\right)
\]
yields the relation

\[
\frac{\frac{\partial f}{\partial z_{2}}A}{f}\ \D z_{1}\wedge \D z_{2}=\frac{1}{f}\, \D z_{1}\wedge \D z_{2},
\]
which implies $A=\frac{1}{\frac{\partial f}{\partial z_{2}}}$ and
consequently $\omega'=\frac{\D z_{1}}{\frac{\partial f}{\partial z_{1}}}$.\end{example}
\begin{remark}
Another way to realize the Poincar\'e residue map is as integration
over a tube $\tau(X)$ along the hypersurface $X$. The map $\text{Res}\colon\ H^{n}(\mathbb{P}^{n}\setminus X)\mapsto H^{n-1}(X)$
is given by 
\[
\omega\mapsto\omega'=\frac{1}{2\pi \i}\int_{\tau(X)}\omega.
\]

\end{remark}

\subsubsection{Higher Order Poles and Reduction of Pole Order }

In this section we will generalize the residue map to rational forms
with higher order poles in order to later more easily find Picard-Fuchs
equations. Let $\mathbb{P}^{n}$ have coordinates $[x_{0}:\ldots:x_{n}]$
and $J\in\mathcal{J}=\left\{ (j_{1},\ldots,j_{k})\colon j_{1}<j_{2}<\ldots<j_{k}\right\} $.
Consider a rational $k$-form on $\mathbb{C}^{n+1}$ given by 
\[
\phi=\frac{1}{B(x)}\sum_{J}A_{J}(x)\D x_{J},
\]
where $x=(x_{0},\ldots,x_{n})$, $\D x_{J}=\D x_{j_{1}}\wedge \D x_{j_{2}}\wedge\ldots\wedge \D x_{j_{k}}$,
and $A_{J},B$ are homogeneous polynomials. By \cite{GRIF1}, this $k$-form comes from
a $k$-form on $\mathbb{P}^{n}$ if and only if $\text{deg}\, B(x)=\text{deg}\, A_{J}(x)+k$
and $\theta(\phi)=0$, where $\theta:=\sum_{i=0}^{n}x_{i}\frac{\partial}{\partial x_{i}}$
is the Euler vector field. This fact allows us to express
rational forms on $\mathbb{P}^{n}$ in a way suitable for later calculations. 
\begin{lemma}
\label{thm:RationalFormsPn}Rational $(n+1-l)$-forms on $\mathbb{P}^{n}$
may all be written as 
\[
\omega=\frac{1}{B(x)}\sum_{J\in\mathcal{J}}\left[(-1)^{\sum_{i=1}^{l}j_{i}}\left(\sum_{i=1}^{l}(-1)^{i}x_{j_{i}}A_{j_{1}\ldots\hat{j_{i}}\ldots j_{l}}(x)\right)\right]\D x_{\hat{J}},
\]
where
$\text{deg}\, B=\text{deg}\,
A_{j_{1}\ldots\hat{j_{i}}\ldots j_{l}}+(n+2-l)$,
and $\D x_{\hat{J}}$ denotes the $(n+1-l)$-form with $\D x_{j}$
omitted if $j\in J$
\end{lemma}
\begin{proof}
\smartqed
This is \cite[Theorem 2.9]{GRIF1}. 
\qed\end{proof}
Let $X$ be a nonsingular hypersurface in $\mathbb{P}^{n}$ given
by the vanishing set $\{Q(x)=0\}$ of a homogeneous polynomial. A
rational $n$-form with a pole along $X$ is then written as 
\[
\omega=\frac{P(x)}{Q(x)}\,\Omega,
\]
where 
\[
\Omega=\sum_{j=0}^{n}(-1)^{i}x_{i}\ \D x_{0}\wedge\ldots\wedge\widehat{\D x_{i}}\wedge\ldots\wedge \D x_{n}
\]
and $\text{deg}\, Q=\text{deg}\, P+(n+1).$ We can assume that $Q(x)=0$
is the minimal defining equation for $X$ so that 
\[
\omega=\frac{P(x)}{Q(x)^{k}}\,\Omega,
\]
where $P$ and $Q$ are relatively prime and $\deg P=k\deg Q-(n+1)$.
In this case we say that $\omega$ has a \emph{pole of order $k\geq1$}
along $X$. Let
\[
\cR:=\left\{ \left.\frac{P\Omega}{Q^{k}}\ \right|\, \deg P=k\deg Q-(n+1)\right\} 
\]
be the set of all rational $n$-forms with a pole along $X$. By \cite{GRIF2},
there is an isomorphism between $\cR$ modulo exact forms and $H^{n}(\P^{n}\setminus X;\C)$.

\begin{definition}
  Let $Q(x) \in \C[x_1 ,\ldots, x_n]$ be
  a polynomial. The \emph{Jacobian ideal} of $Q$ is given by
  $$
  J(Q) = \langle\partial Q/\partial x_{0},\ldots,\partial Q/\partial x_{n}\rangle.
  $$
\end{definition}

\begin{proposition}
\label{prop:ReductionPoleOrder}We can reduce the order of the pole
of $\omega=\frac{P(x)}{Q(x)^{k}}\,\Omega$ from $k\geq2$ to $k-1$
by adding an exact form if and only if $P$ is in the Jacobian ideal $J(Q)$.
\end{proposition}
\begin{proof}
\smartqed
A rational $(n-1)$-form $\varphi$ with a pole of order $k-1$ along
$X$ can by Lemma~\ref{thm:RationalFormsPn} be written as 
\[
\varphi=\frac{1}{Q(x)^{k-1}}\sum_{i<j}(-1)^{i+j}\left(x_{i}A_{j}(x)-x_{j}A_{i}(x)\right)\D x_{\widehat{i,j}}.
\]
A brief calculation shows that 
\begin{equation}
\D \varphi=\left[\frac{(k-1)\sum_{j=0}^{n}\left(A_{j}(x)\frac{\partial Q(x)}{\partial x_{j}}\right)}{Q(x)^{k}}-\frac{\sum_{j=0}^{n}\frac{\partial A_{j}(x)}{\partial x_{j}}}{Q(x)^{k-1}}\right]\Omega,\label{eq:ReductionFormula-1}
\end{equation}
which after rearranging the terms is equivalent to 
\[
\frac{R(x)}{Q(x)^{k-1}}\Omega=\omega+\D \varphi
\]
for some polynomial $R(x)$, proving the result.
\qed 
\end{proof}
Let 

\[
\cJ:=\left\{ \frac{P\Omega}{Q^{k}}\mid P\in J(Q)\right\} \subset\cR,
\]
be the forms whose pole order can be reduced by an exact form. Then
there is a natural filtration of $H^{n}(\P^{n}\setminus X;\C)$ by
pole order, 
\[
B_{1}\to B_{2}\to\ldots\to B_{i}\to B_{i+1}\to\ldots
\]
where 
\[
B_{i}=\left(\frac{\cR}{\cJ}\right)_{i}:=\left\{ \left[\frac{P\Omega}{Q^{i}}\right]\mid\deg(P)=i\deg(Q)-(n+1)\right\} .
\]
Furthermore, by a theorem due to Macaulay \cite[Theorem 4.11]{GRIF1},
for any $P\Omega/Q^{m}$ such that $Q$ is nonsingular and $\deg\, P\geq(n+1)(\deg\, Q-2)$,
we have that $P\in J(Q)$. This means that the filtration stabilizes,
since this inequality is satisfied for $m>n$. In other words, we
can find a vector space basis $\mathcal{B}_{1}\cup\ldots\cup\mathcal{B}_{n}$
of $H^{n}(\mathbb{P}^{n}\setminus X;\C)$, where $\mathcal{B}_{i}$
is a basis of $B_{i}$ consisting of forms with a pole of order $i$.
\begin{remark}
\label{rem:griffithsTransversality}
If $Q$ and some coefficient of $P$ depend on a parameter $\psi$,
and we denote $\D f/\D\psi=f'$ for any polynomial $f$, then 
\[
\frac{\D}{\D\psi}\left(\frac{P\Omega}{Q^{l}}\right)=\frac{(QP'-lPQ')\Omega}{Q^{l+1}}.
\]
In other words, pole order increases by one when differentiating with
respect to $\psi$. It is shown in \cite{GRIF2} that $B_{l}$ can be
identified with the Hodge filtration $F^{n-l}PH^{n-1}(X)$, in which
case the above equation is a manifestation of Griffiths transversality. For more
details, see \cite[Section 10.2.2]{VOIS1}.
\end{remark}
\begin{definition}
Fix an $(n-1)$-cycle $\gamma$ on $X$. The generalized \emph{residue
map} 
\[
\text{Res}\colon\ H^{n}(\mathbb{P}^{n}\setminus X;\C)\to PH^{n-1}(X)
\]
is determined by the relation 
\begin{equation}
\frac{1}{2\pi \i}\int_{\tau(\gamma)}\frac{P}{Q^{l}}\Omega=\int_{\tau}\text{Res}\left(\frac{P}{Q^{l}}\Omega\right),\label{eq:ResidueAsIntegral}
\end{equation}
where $\tau(\gamma)$ is a tube around $\gamma$, and $PH^{n-1}(X)$
is the primitive cohomology of\textbf{ $X$}. If $H$ represents a
hyperplane class, primitive cohomology is defined as 
\[
PH^{n-1}(X)=\{\eta\in H^{n-1}(X;\C)\mid\eta\cdot H=0\}.
\]
\end{definition}
The residue map is surjective in general, and in the case that $n-1$
is odd, primitive cohomology captures all of the cohomology of $X$ 
(for a proof, see \cite{GRIF1}). Therefore, if we are working with an odd-dimensional
one-parameter family of hypersurfaces $X_\psi \colon \{ Q = 0 \}$,
as is the case in equation (\ref{eq:quintic-def}), then $H^{n-1}(X_\psi;\,\C)$
has a basis of residues. Moreover, we have that
\[
\frac{\D ^{k}}{\D \psi^{k}}\int_{\gamma}\text{Res}\left(\frac{P\Omega}{Q^{l}}\right)=\frac{\D ^{k}}{\D \psi^{k}}\left(\frac{1}{2\pi \i}\int_{\tau(\gamma)}\frac{P\Omega}{Q^{l}}\right)=\frac{1}{2\pi \i}\int_{\tau(\gamma)}\frac{\D ^{k}}{\D \psi^{k}}\left(\frac{P\Omega}{Q^{l}}\right).
\]
So, in order to find relations amongst $\frac{\D ^{k}}{\D \psi^{k}}\text{Res}\left(\frac{P\Omega}{Q^{l}}\right)$, we can work with meromorphic forms on $\P^{n}$.

\begin{remark}
If $n=1$, the residue map is the familiar contour integral. For instance,
for a one form $\frac{p(z)}{q(z)}\D z$ with $p,q\in\mathbb{C}[z]$
we have 
\[
\int_{\Gamma}\frac{p(z)}{q(z)}\D z=\frac{1}{2\pi \i}\left[\sum_{P_{j}}\text{Res}_{P_{j}}\left(\frac{p(z)}{q(z)}\right)\right],
\]
where $\Gamma$ encircles all the poles $P_{j}$ of $\frac{p(z)}{q(z)}.$ 
\end{remark}

\subsection{Determining Picard-Fuchs Equations\label{sec:Determining-Picard-Fuchs-Eq}}

By this point we have established sufficient background material to
determine Picard-Fuchs equations for one-parameter families of hypersurfaces
in several ways.

\subsubsection{The Griffiths-Dwork Method}

Let $X_{\psi}\subset\P^{n}$ be an element of a one-parameter family
of hypersurfaces parameterized by $\psi$. Suppose that $X_{\psi}$
is given by $\{Q=0\}$ and choose a form $P\Omega/Q$ whose residue
is a holomorphic $(n-1)$-form $\omega(\psi)$ on $X_{\psi}$. Finding
the Picard-Fuchs equation satisfied by the period $\int_{\gamma}\omega(\psi)$
amounts to finding a relation between $\omega(\psi)$ and its derivatives.
The description of forms on $X_{\psi}$ as residues of meromorphic
forms on $\P^{n}$ gives rise to the following algorithm for finding
Picard-Fuchs equations called the \emph{Griffiths-Dwork method}, also described in \cite{COKA,
GAHRS}.
\begin{enumerate}
\item Find a basis $B$ of meromorphic differentials for $H^{n}(\P^{n}\setminus X_{\psi};\C)$.
This amounts to finding a basis for the ring $\C(\psi)[x_{1},\ldots,x_{n}]/J(Q)$,
where $\C(\psi)$ emphasizes that coefficients are rational functions
in $\psi$.
\item Starting with a form $P\Omega/Q$ as above, calculate $|B|$ of its
derivatives with respect to $\psi$ and express them in terms of forms
in the basis and forms with numerators in $J(Q)$. Pole order increases
with differentiation due to Remark~\ref{rem:griffithsTransversality}, so use Proposition
\ref{prop:ReductionPoleOrder} to reduce the pole order.
\item The $|B|+1$ forms obtained from $\omega$ and its derivatives must
have a relation between them. This is the Picard-Fuchs equation satisfied
by $\omega(\psi)$. \end{enumerate}
\begin{example}
\label{exa:PFE-for-mirror-of-quintic}We follow \cite{COKA} to illustrate
the Griffiths-Dwork method on the mirror $\cW$ of the one-parameter
family $\cM$ of quintic threefolds whose elements are given by

\[
X_{\psi}\colon\left\{ Q:=\sum_{i=1}^5 x_{i}^{5}-5\psi x_{1}x_{2}x_{3}x_{4}x_{5}=0\right\} .
\]
Let us roughly describe the mirror construction. Let $\eta_{i}$ be
a fifth root of unity, and define $\mathcal{G}$ be the group of diagonal
automorphisms 
\[
g\colon(x_{1},x_{2},x_{3},x_{4},x_{5})\mapsto(\eta_{1}x_{1},\eta_{2}x_{2},\eta_{3}x_{3},\eta_{4}x_{4},\eta_{5}x_{5})
\]
which preserve the holomorphic $3$-form on $X_{\psi}$, modulo those
that come from the scaling action of projective space. The mirror
family $\cW$ is then given by the resolution of singularities of
the quotient $\cM/\mathcal{G}$. For a generic pair $M\in\cM$
and $W\in\cW$, Hodge numbers are exchanged according to $h^{p,q}(M)=h^{3-p,q}(W)$
by mirror symmetry. So, since $h^{i,i}(W)=1$, we have that $b_{3}(W)=4$.
Moreover, cohomology of $\cW$ contains the $\mathcal{G}$-invariant
cohomology of $\cM$.

We therefore choose residues of four meromorphic $4$-forms $\omega_{1},\ldots,\omega_{4}$
that are invariant under $\mathcal{G}$. This will give a basis for
the cohomology of the mirror family. Specifically, for any $l\geq1$
we define $P_{l}=(-1)^{l-1}(l-1)!\psi^{l}(\prod_{i=1}^{5}x_{i})^{l-1}$
and $\omega_{l}=P_{l}\Omega/Q^{l}$. Our goal is to find the Picard-Fuchs
equation of $\Res(\omega_{1})$, i.e., the relation between derivatives
of $\omega_{1}$ with respect to $\psi$. It is convenient to define
$w=\psi^{-5}$ and differentiate using the operator $\vartheta_{w}:=w\frac{\D }{\D w}=-\frac{1}{5}\psi\frac{\D }{\D \psi}$.
We have that 
\[
\vartheta_{w}\omega_{l}=-\frac{l}{5}\omega_{l}+\omega_{l+1},
\]
which after repeated application to $\omega_{1}$ yields
\begin{equation}\label{eq:delta_omega_1}
\mat{
\omega_{1}\\
\vartheta_{w}\omega_{1}\\
\vartheta_{w}^{2}\omega_{1}\\
\vartheta_{w}^{3}\omega_{1}
}
=
\mat{
1 & 0 & 0 & 0\\
-\frac{1}{5} & 1 & 0 & 0\\
\frac{1}{25} & -\frac{3}{5} & 1 & 0\\
-\frac{1}{125} & \frac{7}{25} & -\frac{6}{5} & 1
}
\mat{
\omega_{1}\\
\omega_{2}\\
\omega_{3}\\
\omega_{4}
}.
\end{equation}
We need to differentiate one more time in order to get a non-trivial
relation. Since forms can be written in terms of the basis, and $\vartheta_{w}^{4}\omega_{1}$
has a pole of order 5 by Remark~\ref{rem:griffithsTransversality}, it follows that
\[
\vartheta_{w}^{4}\omega_{1}=c_{1}(\psi)\omega_{1}+\ldots+c_{4}(\psi)\omega_{4}+\frac{\left(\sum_{i}A_{i}B_{i}\right)\Omega}{Q^{5}}
\]
for some $c_{j}(\psi)\in\C(\psi)$, $A_{i}(x)\in\C[x_{1},\ldots,x_{5}]$,
where $\{B_{i}\}$ constitutes a Gr\"obner basis for $J(Q)$. We can
reduce the pole order of the last term using Proposition \ref{prop:ReductionPoleOrder},
and again express the lower order form in terms of the basis $\{\omega_{i}\}$.
These calculations can be done using a computer (see \cite{GAHRS} for
details and source code), and the result is the Picard-Fuchs equation
(\ref{eq:quintic-PFE}). 
\end{example}

\begin{example}
The quintic threefold family (\ref{eq:quintic-def}) can be seen as
a deformation of the Fermat quintic $x_{1}^{5}+\ldots+x_{5}^{5}$. This polynomial belongs
to a larger class of \emph{invertible polynomials}. A quasi-homogeneous
polynomial 
\[
G(x)=\sum_{i=1}^{n}c_{i}\prod_{j=1}^{n}x_{i}^{a_{ij}}
\]
with (reduced) weights $(q_{1},\ldots,q_{n})$ is invertible if the
exponent matrix $(a_{ij})$ is invertible and the ring $\C[x]/J(G)$ has
a finite basis. In general, the zero set of an invertible polynomial
defines a variety in a weighted projective space
$\P(q_{1},\ldots,q_{n})$, which can be realized as the quotient of the
usual projective space $\P^{n-1}$ by an abelian group action (for
more details about varieties in weighted projective space, see
\cite{DOLG}). We can obtain a one-parameter family from 
polynomials such as $G(x)$ via
\[
F(x)=G(x)+\psi\prod_{i=1}^{n}x_{i}.
\]
Elements of this family define Calabi-Yau hypersurfaces
if $\sum q_{i}=\deg G$ by \cite{DOLG}, so such deformations provide a
large class examples of one-parameter Calabi-Yau families. A
combinatorial method for calculating the Picard-Fuchs equations of
such families based on the Griffiths-Dwork method is presented in
\cite{GAHRS}. For instance, the family of K3 surfaces
\[
\left\{ x_{1}^{8}+x_{2}^{4}+x_{1}x_{3}^{3}+x_{4}^{3}+\psi\prod x_{i=1}^{3}x_{i}^{4}=0\right\} 
\]
 in $\P(3,6,7,8)$ has Picard-Fuchs equation $\cL f=0$, where for
$\vartheta_{\psi}=\psi\frac{\partial}{\partial\psi}$ we have
\begin{align*}
\cL:&=\psi^{12}\vartheta_{\psi}^{3}(\vartheta_{\psi}+3)(\vartheta_{\psi}+6)(\vartheta_{\psi}+9)\\
&\qquad -2^{8}3^{9}(\vartheta_{\psi}-1)(\vartheta_{\psi}-2)(\vartheta_{\psi}-5)(\vartheta_{\psi}-7)(\vartheta_{\psi}-10)(\vartheta_{\psi}-11).
\end{align*}
\end{example}

\subsection{Finding the Periods\label{sec:Finding-the-Periods}}

In this section we describe how to find series solutions to the Picard-Fuchs
equation (\ref{eq:quintic-PFE}). The solutions correspond to periods
of the mirror, and thus by solving the equation we obtain a series
description of the periods.

\subsubsection{Hypergeometric Series}

Let $Q$ be the defining polynomial of the quintic threefold $X_{\psi}$
in equation (\ref{eq:quintic-def}). It turns out that it is possible
to directly calculate the period on $X_{\psi}$ with respect to 
\begin{equation}
5\psi\frac{x_{5}\D x_{1}\D x_{2}\D x_{3}}{\frac{\partial Q}{\partial x_{4}}}.\label{eq:fundamental-form}
\end{equation}

\begin{example}
\label{exa:ell-curve-direct-period-calc}We will show here the analogous
calculation on the Fermat family of elliptic curves $Z_{\psi}\colon\{F_{\psi}=0\}$
defined in Example \ref{expl:FermatResidueCalculation} which can
be applied, mutatis mutandis, to $X_{\psi}$. The latter appears in
\cite{CANMS}. Denote by $\gamma_{i}$ the cycle on $\{Q=0\}$ determined
by $|x_{i}|=\delta$ for some small $\delta$, and consider 
\[
\pi_{0}(\psi)=3\psi\frac{1}{2\pi \i}\int_{\gamma_{1}}\frac{x_{3}\D x_{1}}{\frac{\partial F_{\psi}}{\partial x_{2}}}.
\]
Since $1=\frac{1}{2\pi \i}\int_{\gamma_{3}}\frac{\D x_{3}}{x_{3}}$ and
\[
\text{Res}\left(\frac{\D x_{2}}{f(x)}\right)=\frac{1}{2\pi \i}\int_{\gamma_{2}}\frac{\D x_{2}}{F_{\psi}}=\frac{-1}{\frac{\partial F_{\psi}}{\partial x_{2}}},
\]
we have 
\begin{align*}
\pi_{0}(\psi)= & -3\psi\frac{1}{(2\pi \i)^{3}}\int_{\gamma_{1}\times\gamma_{2}\times\gamma_{3}}\frac{\D x_{1}\D x_{2}\D x_{3}}{F_{\psi}(x)}\\
= & -3\psi\frac{1}{(2\pi \i)^{3}}\int_{\gamma_{1}\times\gamma_{2}\times\gamma_{3}}\frac{\D x_{1}\D x_{2}\D x_{3}}{x_{1}^{3}+x_{2}^{3}+x_{3}^{3}-3\psi x_{1}x_{2}x_{3}}\\
= & 3\psi\frac{1}{(2\pi \i)^{3}}\int_{\gamma_{1}\times\gamma_{2}\times\gamma_{3}}\frac{\D x_{1}\D x_{2}\D x_{3}}{3\psi\ x_{1}x_{2}x_{3}}\frac{1}{1-\frac{x_{1}^{3}+x_{2}^{3}+x_{3}^{3}}{3\psi x_{1}x_{2}x_{3}}}\\
= & \frac{1}{2\pi \i}\sum_{n=0}^{\infty}\int_{\gamma_{1}\times\gamma_{2}\times\gamma_{3}}\frac{\D x_{1}\D x_{2}\D x_{3}}{x_{1}x_{2}x_{3}}\frac{1}{(3\psi)^{n}}\frac{(x_{1}^{3}+x_{2}^{3}+x_{3}^{3})^{n}}{(x_{1}x_{2}x_{3})^{n}},
\end{align*}
with the expansion performed for large enough $\psi$. We now wish
to evaluate the integral using residues. The integral for each $n$
is a rational function in the variables $x_{i}$ and therefore vanishes
for all powers of $x_{i}$ except $-1$. This happens when the term
$(x_{1}x_{2}x_{3})^{n}$ occurs in the expansion of $(x_{1}^{3}+x_{2}^{3}+x_{3}^{3})^{n}$.
To see when this is the case, consider 
\[
(x_{1}^{3}+x_{2}^{3}+x_{3}^{3})^{n}=\sum_{k_{1}+k_{2}+k_{3}=n}{n \choose k_{1},k_{2},k_{3}}x_{1}^{3k_{1}}x_{2}^{3k_{2}}x_{3}^{3k_{3}},
\]
and note that we want $k_{1}=k_{2}=k_{3}=k$ so that $n=3k$. The
coefficient we need is then ${3k \choose k,k,k}=\frac{(3k)!}{(k!)^{3}}$
and the expression for $\pi_{0}(\psi)$ becomes 
\begin{equation}
\pi_{0}(\psi)=\sum_{k=0}^{\infty}\frac{(3k)!}{(k!)^{3}}\frac{1}{(3\psi)^{3k}}\label{eq:CandelasHypergeometricFermatExpansion}
\end{equation}
for $\psi$ large enough.
\end{example}
Similarly, the integral $\varpi_0(\psi)$ of (\ref{eq:fundamental-form}) on $X_{\psi}$
is given by equation (\ref{eq:fundamental-period-quintic}), i.e., 
\[
\varpi_{0}(\psi)=\sum_{m=0}^{\infty}\frac{(5m)!}{(m!)^{5}}\lambda^{m},
\]
where $\lambda=\frac{1}{(5\psi)^{5}}$. This gives one period of $X_{\psi}$
and, as we will see in a moment, a solution of $\cL f=0$ for $\cL$ defined in (\ref{exa:PFE-for-mirror-of-quintic}).
What about the other solutions? Recall that 
\begin{equation}
f(z)=\sum_{k}C(k)z^{k}\label{eq:generic-hypergeometric-series}
\end{equation}
is a \emph{(generalized) hypergeometric series} if the ratio of consecutive
terms is a rational function of $k$, 
\begin{equation}
\frac{C(k+1)}{C(k)}=c\frac{(k+a_{1})\ldots(k+a_{p})}{(k+b_{1})\ldots(k+b_{q})(k+1)},\label{eq:generic-hypergeometric-ratio}
\end{equation}
where $c$ is a constant \cite{DWO}. In the case of the Fermat family
of elliptic curves, we indeed have
\begin{equation}
\frac{(3(k+1))!/((k+1)!)^{3}}{(3k)!/(k!)^{3}}=\frac{3(k+\frac{2}{3})(k+\frac{1}{3})}{(k+1)(k+1)}.\label{eq:ratio-for-fermat-elliptic}
\end{equation}
The standard notation for a hypergeometric \emph{function} given by
(\ref{eq:generic-hypergeometric-series}) is 
\[
f(z)={}_{p}F_{q}(a_{1},\ldots,a_{p};\ b_{1},\ldots,b_{q};\ z),
\]
in which case $f(z)$ satisfies the \emph{hypergeometric differential
equation} 
\begin{equation}
\left[\vartheta_{z}\prod_{i=1}^{q}(\vartheta_{z}+b_{i}-1)-z\prod_{i=1}^{p}(\vartheta_{z}+a_{i})\right]f(z)=0,\label{eq:generic-hypergeometric-diffyq}
\end{equation}
where $\vartheta_{z}:=z\frac{\D}{\D z}.$ The Picard-Fuchs equation (\ref{eq:quintic-PFE})
is hypergeometric, where one solution is given by $\varpi_{0}=\,_{4}F_{3}(\frac{1}{5},\frac{2}{5},\frac{3}{5},\frac{4}{5};\,1,1,1;\,\psi^{-5})$.
We will next explain how to find the remaining solutions.

\subsubsection{Frobenius Method}

Hypergeometric differential equations can be solved using the Frobenius
method. We will illustrate the basic technique for the Picard-Fuchs
equation (\ref{eq:quintic-PFE}). We already have the series description
(\ref{eq:fundamental-period-quintic}) of one solution around $\lambda=0$

\[
\varpi_{0}=\sum_{m=0}^{\infty}\frac{\Gamma(5k+1)}{\Gamma(k+1)^{5}}\lambda^{m},
\]
where $\lambda=\frac{1}{(5\psi)^{5}}$. So our goal is to obtain the
remaining three solutions of the differential equation $\cL f(z)=0$,
where $\cL$ is defined in equation (\ref{eq:quintic-PFE}). Before
we do so, we remark that there is a more systematic way of finding
the first solution of a hypergeometric differential equation than
the direct calculation of the integral $\varpi_{0}$. Since it is
not critical to what follows, we illustrate with a quick example.
\begin{example}
The differential equation satisfied by the period $\pi_{0}$ of Example
\ref{expl:FermatResidueCalculation} is by equation (\ref{eq:ratio-for-fermat-elliptic})
given by

\begin{equation}
\cL f(z)\equiv\left[\vartheta_{z}^{2}-z(\vartheta_{z}+\frac{1}{3})(\vartheta_{z}+\frac{2}{3})\right]f(z)=0\label{eq:fermat-elliptic-curve-diffq}
\end{equation}
in terms of $z=1/(3\psi)^{3}$. We now make the ansatz 
\begin{equation}
f(z)=z^{c}+\sum_{k=1}^{\infty}a_{k}z^{k+c}\label{eq:fermat-elliptic-ansatz}
\end{equation}
for the solution around the regular singular point $z=0$. Applying
the differential equation and setting coefficients to zero, we obtain
\[
c^{2}=0\quad\text{and}\quad(k+c)^{2}a_{k}-\left((k-1+c)^{2}+(k-1+c)+\frac{2}{9}\right)a_{k-1}=0
\]
from the $z^{c}$ term and the $z^{k+c}$ terms for $k\geq1$, respectively.
The first equation is called the \emph{indicial equation}, and implies
that $c=0$. Using this in the second equation gives for $k\geq1$
\[
a_{k}=\frac{(\frac{1}{3}+k-1)(\frac{2}{3}+k-1)}{k^{2}}a_{k-1}.
\]
We can then iterate to obtain the solution 
\[
f_{1}(z)=\sum_{k=0}^{\infty}\frac{(\frac{1}{3})_{k}(\frac{2}{3})_{k}}{(1)_{k}^{2}}z^{k},
\]
where $(a)_{k}=\frac{\Gamma(a+k)}{\Gamma(a)}=a(a+1)\ldots(a+k-1)$
is the Pochhammer symbol. It is easily checked that this solution
is equivalent to (\ref{eq:CandelasHypergeometricFermatExpansion}).
There is, of course, a second solution of (\ref{eq:fermat-elliptic-curve-diffq}).
The indicial equation is of degree two (which is implied by the fact
that $z=c$ is a \emph{regular }singular point), but has a repeated
root and cannot be used again as above. The idea is to show that $\left. \frac{\partial f}{\partial c}\right|_{c=0}$
is a solution, the analog of which we will tackle for the quintic
directly.
\end{example}
We now return to the case of the quintic threefold family (\ref{eq:quintic-def}).
Define 
\[
\varpi(\lambda,s)=\sum_{k=1}^{\infty}\frac{\Gamma(5(k+s)+1)}{\Gamma(k+s+1)}\lambda^{k+s}
\]
and note that $\varpi_{0}(\lambda)=\varpi(\lambda,0)$. A direct calculation
shows that $\cL\varpi(\lambda,s)=s^{4}\lambda^{s}+\cO(s^{5})$, from
which it follows that for $0\leq i\leq3$ we have
\[
\left. \frac{\partial^{i}}{\partial s^{i}}\cL\varpi(\lambda,s)\right|_{s=0}=\cL\left. \frac{\partial^{i}}{\partial s^{i}}\varpi(\lambda,s)\right|_{s=0}=0.
\]
This gives us a set of solutions, 

\[
\varpi_{i}(\lambda)=\left. \frac{\partial^{i}}{\partial s^{i}}\varpi(\lambda,s)\right|_{s=0}\,.
\]
To describe them more explicitly and see that they are linearly independent, let
\begin{equation}
  a_{k}(s):=\frac{\Gamma(5(k+s)+1)}{\Gamma(k+s+1)}
  \quad\text{and}\quad
  g_{i}(z)=\sum_{k=0}^{\infty} \left. \frac{\partial^{i}a_{k}(s)}{\partial s^{i}}\right|_{s=0}\,\lambda^{k}.\label{eq:quintic-log-term-solutions}
\end{equation}
We calculate
\begin{align*}
\varpi_{1}(\lambda) & =\left(\sum_{k=0}^{\infty}a_{k}\lambda^{k}\right)\log\lambda+\sum_{k=0}^{\infty}\left. \frac{\partial a_{k}(s)}{\partial s}\right|_{s=0}\,\lambda^{k}\\
 & =\varpi_{0}(\lambda)\log\lambda+g_{1}(\lambda).
\end{align*}
Iterating, we obtain a full set of solutions. Namely, for $0\leq i\leq3$
the solutions are
\begin{equation}
\varpi_{i}(\lambda)=\sum_{j=0}^{i}{i \choose j}g_{j}(\lambda)(\log\lambda)^{i-j}.\label{eq:general-quintic-solutions}
\end{equation}
In the case of the family given in (\ref{eq:quintic-PFE}), these solutions correspond to the periods by \cite{CANMS}.

\section{Point Counting}

In this chapter we will show how to obtain the expression (\ref{eq:quintic-period-nr-points})
for the number of points on a quintic threefold (\ref{eq:quintic-def})
defined over a finite field $k=\F_{p}$ of characteristic $p$. The
basic idea is to use $p$-adic character formulas to mimic the behavior
of period integrals via $p$-adic analysis techniques. We will also
explain how to calculate the zeta functions of several
varieties over finite fields, and discuss the relationship of zeta
functions and mirror symmetry.

\subsection{Character Formulas\label{sec:Character-Formulas}}

Let us first establish some basics about characters of finite groups.
Let $K=\C$ or $\C_{p}$, let $G$ be a nontrivial finite abelian
group, and take a non-trivial character $\chi \colon G\to K$.
We then have that
\[
\sum_{x\in G}\chi(x)=0
\]
since for $y\in G$ such that $\chi(y)\neq1$ we have $\chi(y)\sum_{x\in G}\chi(x)=\sum_{x\in G}\chi(x)$,
and so $(\chi(y)-1)\sum_{x\in G}\chi(x)=0.$ It is also easy to see
that 
\begin{equation}
\sum_{\chi\in\widehat{G}}\chi(x)=\begin{cases}
0, & \text{if }x\neq1;\\
|G|, & \text{if }x=1,
\end{cases}\label{eq:summation-multiplicative-char-over-grp}
\end{equation}
and 
\[
\chi(x^{-1})=\chi(x)^{-1}=\overline{\chi(x)},
\]
where if the character $\chi$ maps into $\C_p$, we define
$\overline{\chi(x)} = \chi^{-1}(x)$. Our goal is to count points on
hypersurfaces defined over $k=\F_{q}$, where $q=p^{n}$ with
coordinates in some $k_{r}=$degree $r$ extension of $k$. 
Denote the trivial character by $\varepsilon$. If
$\chi\colon\knus\to K$ is a multiplicative character we can define
\[
\chi(0)=\begin{cases}
0, & \chi\neq\varepsilon;\\
1, & \chi=\epsilon,
\end{cases}
\]
so that we can consider it as a homomorphism 
\[
\chi\colon\knu\to K.
\]
Fixing a non-trivial \emph{additive} character $\psi\colon\knu\to(K,\times)$,
we define the \emph{Gauss sum} 
\[
g(\chi)=\sum_{x\in k_{r}}\chi(x)\psi(x)
\]
which for non-trivial $\chi$ equals $g_{0}(\chi):=\sum_{x\in k_{r}^{*}}\chi(x)\psi(x)$
and otherwise 
\[
g(\varepsilon)=\sum_{x\in k_{r}}\psi(x)=0\quad\text{and}\quad g_{0}(\varepsilon)=-1,
\]
since $\psi$ is non-trivial. 

Gauss sums $g_{0}$ are proportional to Fourier transforms: consider
$\psi\colon\knus\to K$ as a $K$-valued function on $\knus$, and let
the Fourier transform of $f$ to be the $K$-valued function on the
group $\widehat{\knus}$ of multiplicative characters
$\chi \colon \knus \to K$ given by
\[
\widehat{f}(\conj{\chi})=\frac{1}{q^{r}-1}\sum_{x\in\knus}\psi(x)\chi(x).
\]
We also get Fourier inversion, i.e., we can express $f(x)$ in terms of
characters. Consider the sum over all multiplicative characters $\chi\colon\knu^{*}\to K$
for any $x\neq0$,
\begin{align*}
\sum_{\chi}g_{0}(\conj{\chi})\chi(x) & =\sum_{\chi}\left(\sum_{y\in\knus}\conj{\chi}(y)\psi(y)\right)\chi(x)\\
 & =\sum_{\chi}\sum_{y\in\knu^{*}}\chi(y^{-1}x)\psi(y)\\
 & =\sum_{y\in\knu^{*}}\psi(y)\sum_{\chi}\chi(y^{-1}x)\\
 & =(q^{r}-1)\psi(x).
\end{align*}
Therefore, for all $x\neq0$ we have 
\begin{equation}
\psi(x)=\frac{1}{q^{r}-1}\sum_{\chi}g_{0}(\conj{\chi})\chi(x)=\frac{1}{q^{r}-1}\sum_{\chi}g_{0}(\chi)\conj{\chi}(x).\label{eq:char-fourier-inversion}
\end{equation}
This is the Fourier inversion formula for $f=\psi$ (up to an unconventional
choice for what we are conjugating):
\[
f(x)=\sum_{\chi\in\widehat{G}}\widehat{f}(\conj{\chi})\conj{\chi}(x).
\]

\begin{remark}
If $x=0$, then $\psi(0)=1$ and since $\conj{\chi}(0)=0$ unless
$\conj{\chi}=\varepsilon$ we have the right hand side equal to $\frac{1}{q^{r}-1}g_{0}(\varepsilon)=\frac{1}{q^{r}-1}(-1)\neq1$.
So the formula does not hold for $x=0$. This will prove to be a minor
annoyance when counting points.
\end{remark}

\begin{remark}
\label{rem:inverting-gauss-sum-wo-zero}Since $g(\varepsilon)\conj{\varepsilon}(x)=g_{0}(\varepsilon)\conj{\varepsilon}(x)+\varepsilon(0)\psi(0)\conj{\varepsilon}(x)=g_{0}(\varepsilon)\conj{\varepsilon}(x)+1$
and for $\chi\neq\varepsilon$ $g(\chi)=g_{0}(\chi)$ we have
\[
\psi(x)=\frac{1}{q^{r}-1}\left(\sum_{\chi}g(\chi)\conj{\chi}(x)-1\right),
\]
which is harder to work with, even though $g(\varepsilon)=0$ and
$g_{0}(\varepsilon)\neq0$. Therefore, we will be using the Gauss
sums $g_{0}(\chi)$ as opposed to $g(\chi)$.
\end{remark}

\subsubsection{$p$-adic Characters}

We will now construct a concrete multiplicative and additive character
into the $p$-adic numbers $\C_{p}$ to use with the formulas above.
For now, we will use characters from $k=\F_{p}$, leaving finer fields
for later. Given $n=mp^{v}\in\Z$, where $(p,m)=1$ define the \emph{$p$-adic
norm} $|n|_{p}=\frac{1}{p^{v}}$. The completion of $\Z$ with respect
to this norm gives the $p$-adic integers $\Z_{p}$, which can be written
as sequences

\begin{align*}
\mathbb{Z}_{p}:=\underleftarrow{\lim}_{n}\mathbb{Z}/p^{n}\mathbb{Z}= & \left\{ a_{0}+a_{1}p+a_{2}p^{2}+\ldots\mid a_{i}\in[0,p-1]\right\} \\
= & \lim_{n\to\infty}\left(\sum_{i=0}^{n}a_{i}p^{i}\right),
\end{align*}
with the last expression being thought of as giving increasingly better
approximations of the corresponding $p$-adic integer as $n\to\infty$.
Taking the field of fractions gives $\Q_{p}=\text{Frac}(\Z_{p})$
whose algebraic closure $\overline{\Q_{p}}$ is not complete. The
completion of $\overline{\Q_{p}}$ is $\C_{p}$, and is also algebraically
closed. 
\begin{lemma}[Hensel's Lemma]
\label{lemma: hensel's lemma}Suppose that $\overline{f}\in\F_{p}[x]$
and let $f\in\Z_{p}[x]$ be any lift (so that $f\equiv\overline{f}\mod p$).
If $\overline{\alpha}\in\F_{p}$ is a simple root of $\overline{f}$,
then there exists a unique $\alpha\in\Z_{p}$ such that 
\[
F(\alpha)=0\quad\text{and}\quad a\equiv\overline{\alpha}\modp.
\]
\end{lemma}
\begin{proof}
\smartqed
See \cite{KOBLITZ}.
\qed
\end{proof}
\begin{proposition}
\label{definition_proposition: techmuller character} For each $x\in\Fcross$
there is a unique $(p-1)$-st root of unity in $\Z_{p}^{*}$ denoted
$\Teich(x)$ or $T(x)$ such that $T(x)\equiv x\modp$. The map $T\colon\Fcross\to\Z_{p}^{*}$
given by $x\mapsto T(x)$ gives a multiplicative character called
the Teichm\"uller character. 
\end{proposition}
\begin{proof}
\smartqed
The elements of $\Fcross$ are the roots of $\overline{f}(X)=X^{p-1}-1$
since each $x\in\Fcross$ satisfies the equation $\overline{f}(X)=0$.
Let $f(X)=X^{p-1}-1\in\Z_{p}[X]$ be a lift of $\overline{f}(X)$.
By Hensel's Lemma, each $x\in\Fcross$ lifts to a unique $(p-1)$-st
root of unity $T(x)\in\Z_{p}^{*}$ such that $T(x)\equiv x\modp$.
Since a product of roots of unity is still a root of unity, for $x,y\in\Fcross$
we have that 
\[
\left(T(x)T(y)\right)^{p-1}=1
\]
in $\Z_{p}$. Since we also have that 
\[
T(x)T(y)\equiv xy\modp
\]
it must be the case that 
\[
T(x)T(y)=T(xy)
\]
by uniqueness in Hensel's Lemma.
\qed
\end{proof}

We will use an explicit description of $T(x)$ as in \cite{CORV1}.  Let
$x$ denote an integer representative of $x\in\Fcross$. We have that
\[
x^{p-1}=1+\mathcal{O}(p)\text{ in }\Z,
\]
and consequently that 
\[
\left(x^{p-1}\right)^{p}=1+{p \choose 1}\mathcal{O}(p)+\mathcal{O}(p^{2})=1+\mathcal{O}(p^{2})\text{ in }\Z.
\]
By raising both sides of this equation to the $p$-th power repeatedly,
it follows that 
\[
x^{p^{n}(p-1)}=1+\mathcal{O}(p^{n+1}),
\]
which is equivalent to 
\[
x^{p^{n+1}}=x+\mathcal{O}(p^{n+1}).
\]
Define 
\begin{equation}
S(x):=\lim_{n\to\infty}x^{p^{n}}.\label{equation: x powers teichmuller description}
\end{equation}
The character $T$ is uniquely determined by the conditions
$T(x)^{p-1}=1$ and $T(x)\equiv x\modp$ for all $x\in\Fcross$, or
equivalently the conditions $T(x)^{p}=T(x)$ and $T(x)\equiv x\modp$
for all $x\in\Fcross$. Since the expression in
equation~(\ref{equation: x powers teichmuller description}) satisfies
both of these conditions, we conclude that $S(x)=T(x)$ for all
$x\in\Fcross$. In fact, defining $T^i\colon \Fcross \to \Z_p^*$ by
$T^i(x)= T(x)^i$ for $i\in\{0,\ldots,p-2\}$ gives a full set of
characters from $\Fcross$ to $\Z_{p}^{*}$. Note that when applied to
$\chi=T^{i}$, equation
(\ref{eq:summation-multiplicative-char-over-grp}) takes the form
\[
\sum_{x\in\Fcross}T^i(x)=\begin{cases}
0, & \text{if }i\not\equiv0\text{ mod }p-1;\\
p-1, & \text{if }i\equiv0\text{ mod }p-1.
\end{cases}
\]

We now turn to constructing an additive character $\theta\colon\mathbb{F}_{p}\to\mathbb{C}_{p}^{\times}$.
Let $\zeta_{p}$ be a $p$-th root of unity in the $p$-adic numbers
and define
\[
\theta(x)=\zeta_{p}^{T(x)}.
\]
Since for $Z\in\Z_{p}$ we have

\[
T(x+y)=T(x)+T(y)+pZ
\]
it follows that
\[
\theta(x+y)=\theta(x)+\theta(y),
\]
so $\theta$ is indeed an additive character. For the root of unity
we can take $\Theta(x)=\exp\left(\pi(x-x^{p})\right)$ and set $\zeta_{p}:=\Theta(1)$, as shown in 
\cite{KOBLITZ}.

Recalling Remark \ref{rem:inverting-gauss-sum-wo-zero}, we consider
the Gauss sum associated to these characters, 
\[
G_{n}=\sum_{x\in\mathbb{F}_{p}^{\times}}\theta(x)T^{n}(x),
\]
where $n\in\Z$. 
\begin{remark}
\label{rem:p-adic-gamma-gauss-sum}This expression can be thought
of as a $p$-adic analog of the classical Gamma function,
\[
\Gamma(s)=\int_{0}^{\infty}\frac{\D t}{t}t^{s}\eul^{-t},
\]
where we think of $T(x)$ as the analog of $t\mapsto t^{s}$, of
$\theta(x)$ as the analog of $t\mapsto \eul^{-t}$, and of summation
over $\F_{p}^{*}$ as the analog of integration with respect to
the Haar measure $\frac{\D t}{t}$. In fact, relations can be proven
in terms of these characters for $\Gamma(s)$ can also be given for
$G_{n}$. 
\end{remark}
In this setting, formula (\ref{eq:char-fourier-inversion}) is
\begin{equation}
\theta(x)=\frac{1}{p-1}\sum_{m=0}^{p-2}G_{-m}T^{m}(x),\label{eq:gauss_sum}
\end{equation}
and if $p-1\nmid n$, we also obtain the relation 
\begin{equation}
G_{n}G_{-n}=(-1)^{n}p.\label{eq:gauss-sum-identity}
\end{equation}

\subsubsection{Relationship with the Periods}

To actually count points, we note that for any polynomial $P(x)\in\F_{p}[x]$,
we have that

\[
\sum_{y\in\mathbb{F}_{p}}\theta(yP(x))=\begin{cases}
p, & \text{if }P(x)=0;\\
0, & \text{if }P(x)\neq0,
\end{cases}
\]
so that 
\[
\sum_{y\in\mathbb{F}_{p}}\sum_{x\in\mathbb{F}_{p}^{3}}\theta(yP(x))=p\cdot N^{*}(X),
\]
where $N^*(X)$ is the number of points on $X\colon \{ P(x)=0 \}$ with
coordinates in $\mathbb{F}_{p}^{*}$. We will illustrate the use of
this formula using the Fermat family of elliptic curves.
\begin{example}
Let $F_{\psi}=\sum_{i=1}^{3}x_{i}^{3}-3\psi x_{1}x_{2}x_{3}\in\F_{p}[x]$, and let
$N^{*}(Z_\psi)$ denote the number of nonzero $\F_{p}$-rational points
on $Z_{\psi}\colon\{F_{\psi}=0\}$. We have that
\[
p N^{*}(Z_\psi)-(p-1)^{3}=\sum_{y,x_{i}\in\mathbb{F}_{p}^{\times}}\theta(yF_{\psi}(x)),
\]
and by (\ref{eq:gauss_sum}) that 
\begin{align*}
\theta(yF(x)) & =\left(\prod_{i=1}^{3}\theta(yx_{i}^{3})\right)\theta(-3\psi yx_{1}x_{2}x_{3})\\
 & =\frac{1}{p-1}\sum_{m=0}^{p-2}G_{-m}T^{m}\left(-3\psi y\right)\prod_{i=1}^{3}\theta(yx_{i}^{3})T^{m}(x_{i}).
\end{align*}
Therefore, 
\begin{align*}
p N^{*}(Z_\psi)-(p-1)^{3} & =\sum_{\substack{x\in\left(\Fcross\right)^{3}\\
y\in\Fcross
}
}\frac{1}{p-1}\sum_{m=0}^{p-2}G_{-m}T^{m}\left(-3\psi y\right)\prod_{i=1}^{3}\theta(yx_{i}^{3})T^{m}(x_{i})\\
 & =\frac{1}{p-1}\sum_{m=0}^{p-2}G_{-m}T^{m}(-3\psi)\sum_{y\in\Fcross}T^{m}(y)\left(\sum_{w\in\Fcross}\theta(yw^{3})T^{m}(w)\right)^{3},
\end{align*}
where we have used the fact that $\sum_{x\in(\Fcross)^{3}}\theta(0F(x))=(p-1)^{3}$
in the first step, and renamed the variables $x_{i}$ to $w$, since
for each $x_{i}$ the sum is identical. Suppose that $3\nmid(p-1)$.
Since $3$ and $(p-1)$ are relatively prime, there exist $a,b\in\Z$
such that $3a+b(p-1)=1$. In particular, we have that 
\[
3a\equiv1\text{ mod }p-1.
\]
Since $T^{l(p-1)}$ is the identity character for any $l\in\Z$, this
implies that 
\[
T^{m}=T^{3am}=\left(T^{am}\right)^{3}
\]
and so 
\begin{align*}
p N^{*}(Z_\psi)-(p-1)^{3} & =\frac{1}{p-1}\sum_{m=0}^{p-2}G_{-m}T^{m}(-3\psi)\sum_{y\in\Fcross}\left(\sum_{w\in\Fcross}\theta(yw^{3})T^{m}(w)T^{ma}(y)\right)^{3}\\
 & =\frac{1}{p-1}\sum_{m=0}^{p-2}G_{-m}T^{m}(-3\psi)\sum_{y\in\Fcross}\left(\sum_{w\in\Fcross}\theta(yw^{3})T^{ma}(w^{3})T^{ma}(y)\right)^{3}\\
 & =\frac{1}{p-1}\sum_{m=0}^{p-2}G_{-m}T^{m}(-3\psi)\left((p-1)G_{ma}^{3}\right)\\
 & =p\sum_{m=0}^{p-2}\frac{G_{ma}^{3}}{G_{m}}T^{m}(-3\psi)(-1)^{m}.
\end{align*}
To simplify this expression, note that the $m$-th Gauss sum $G_{m}=G_{m+l(p-1)}$
depends only on the class of $m$ modulo $p-1$, since $m$ only appears
in the exponent of $T$ within the sum. Define the map $\phi\in\text{End}\left(\Z/(p-1)\Z\right)$
by $m\mapsto am$. Its inverse is given by $\phi^{-1}\colon k\mapsto3k$,
which allows us to rewrite the expression above as 
\begin{align*}
N^{*}(Z_\psi) & =\sum_{k=0}^{p-2}\frac{G_{k}^{3}}{G_{3k}}T^{3k}(-3\psi)(-1)^{3k}+(p-1)^{3}\\
 & =1+\sum_{k=1}^{p-2}\frac{G_{k}^{3}}{G_{3k}}T^{3k}(3\psi)+(p-1)^{3},
\end{align*}
where we have used the fact that $T(-1)=-1$ and that 
\[
G_{0}=\sum_{x\in\Fcross}\theta(x)=-1.
\]

\end{example}
The analogous calculation for the quintic threefold is performed in
\cite{CORV1}. The result after accounting for points for which $x_{i}=0$
for some coordinate is in the case $5\nmid(p-1)$ given by the expression

\[
N(X_\psi)=1+p^{4}+\sum_{m=1}^{p-2}\frac{G_{m}^{5}}{G_{5m}}\Teich^{-m}(\lambda),
\]
where $\lambda=1/(5\psi)^{5}$. Using (\ref{eq:gauss-sum-identity})
and $-m\mapsto(p-1)-m$, which does not change the expression, we obtain
exactly equation (\ref{eq:quintic-period-nr-points}). If we keep
Remark \ref{rem:p-adic-gamma-gauss-sum} in mind, this allows us to
interpret the number of points $N(X_\psi)$ as the $p$-adic analog
to the period (\ref{eq:fundamental-period-quintic}). We can relate
the number of points to the periods further. Let $g_{i}(z$) be defined
as in equation (\ref{eq:quintic-log-term-solutions}) and $\vartheta=\lambda\frac{\D }{\D \lambda}$.
Then it can be shown \cite[equation (6.1)]{CORV1} that 
\begin{align}
  N(X_\psi) & =\,^{(p-1)}g_{0}(\lambda^{p^{4}})+\left(\frac{p}{1-p}\right)\,^{(p-1)}(\vartheta g_{1})(\lambda^{p^{4}}) \nonumber \\
  &\quad +\frac{1}{2!}\left(\frac{p}{1-p}\right)^{2}\,^{(p-1)}(\vartheta^{2}g_{2})(\lambda^{p^{4}})\nonumber \\
  &\quad\quad + \frac{1}{3!}\left(\frac{p}{1-p}\right)^{3}\,^{(p-1)}(\vartheta^{3}g_{3})(\lambda^{p^{4}}) \nonumber \\
  &\quad\quad\quad+\frac{1}{4!}\left(\frac{p}{1-p}\right)^{4}\,^{(p-1)}(\vartheta^{4}g_{4})(\lambda^{p^{4}})\mod\,
  p^{5}.\label{eq:nr-points-as-all-solns-and-semiperiod}
\end{align}
The last term makes an appearance in a period-like integral called
a \emph{semi-period} in \cite{CORV1}. In particular, analogously to
equation (\ref{eq:general-quintic-solutions}), let 
\[
\varpi_{4}(z)=\sum_{j=0}^{4}{4 \choose j}g_{j}(z)(\log z)^{4-j}.
\]
This expression can also be given as $\int_{\gamma'}\omega$, but where
$\partial\gamma'\neq0$. While this calculation directly demonstrates that
periods calculate the number of points, it does not explain
the \emph{reason} for this phenomenon. More conceptual approaches are 
considered in Section~\ref{sec:Zeta-Functions-and}.

An alternative point of view can be given by realizing period and
semi-period integrals in a different form. Let $Q_{\psi}:=\sum_{i=1}^{5}x_{i}^{5}-5\psi\prod_{i=1}^{5}x_{i}$.
Using the calculation in Example \ref{expl:FermatResidueCalculation}
we can write the fundamental period as
\[
\varpi_{0}=\frac{5\psi}{(2\pi \i)^{4}}\int_{\gamma_{2}\times\ldots\times\gamma_{5}}\frac{x_{1}\D x_{2}\ldots \D x_{5}}{Q_{\psi}(x)},
\]
where $x_{1}$ is kept constant and $\gamma_{i}$ is a circle around
the origin as in the example. Using the integral
\[
\frac{1}{Q_{\psi}}=\int_{0}^{\infty}\eul^{-tQ_{\psi}}dt
\]
we can rewrite this period as 
\begin{align*}
\varpi_{0} & =\frac{5\psi}{(2\pi \i)^{4}}\int_{\bar{\gamma}_{2}\times\ldots\times\bar{\gamma}_{5},t}\ x_{1}\eul^{-tQ_{\psi}(x)}\D x_{2}\ldots \D x_{5}dt\\
 & =\frac{5\psi}{(2\pi \i)^{4}}\int_{\bar{\gamma}_{1}\times\ldots\times\bar{\gamma}_{5}}\ \eul^{-Q_{\psi}(x)}\D x_{1}\ldots \D x_{5}\,,
\end{align*}
where the last step follows from re-absorbing the parameter $t$ into
the variables via $x_{i}\mapsto s^{-\frac{1}{5}}x_{i}$, and where
$\bar{\gamma}_{i}$ is the corresponding change in the domain of
integration. This point of view gives a more algebraic description of
the periods which we now describe in rough terms. We can consider
periods (and semi-periods) of some hypersurface
$\{Q(x_{1},\ldots,x_{n})=0\}$ as integrals of the form
\begin{equation}
\int_\Gamma x^{\vec v}\eul^{-Q}\,\D x_{1}\ldots \D x_{n},\label{eq:period-d-module-form}
\end{equation}
where $x^{\vec v}$ is a monomial in $\C[x_{1},\ldots,x_{n}]$, and
$\Gamma$ is a cycle on $\C^n$ such that $\eul^{-Q}$ goes to zero
sufficiently quickly. Such integrals have a cohomological analog.  In
particular, we can think of the integral
(\ref{eq:period-d-module-form}) as the element $x^{\vec v}$ in the
module $M=\C[x]\,\eul^{-Q}$ of the \emph{Weyl algebra $A_{n}$}
generated by $x_{1},\ldots,x_{n},\partial_{1},\ldots,\partial_{n}$
modulo the relations
\[
[\partial_{i},\partial_{j}]=0,\ [x_{i},x_{j}]=0,\text{ and }[\partial_{i},x_{j}]=\delta_{i,j}\,,
\]
where the action of $\partial_{i}$ is taken to be $\frac{\partial}{\partial x_{i}}-\frac{\partial Q}{\partial x_{i}}$
(capturing the chain rule). Assuming that $Q$ is non-singular, it
can be shown that the algebraic de Rham cohomology of the module $M$
is given by
\[
H(M)=\frac{M}{\partial_{1}M+\ldots+\partial_{n}M}\ \D x_{1}\ldots \D x_{n},
\]
and that integrals of the form (\ref{eq:period-d-module-form}) are
independent of the choice of representative of the cohomology class of
$x^{\vec v}$. The module $M$ can be similarly defined over $\C_p$, in
which case there exists an action of Frobenius on cohomology that can
be used to study arithmetic of the periods. For a more formal
explanation and further details, we refer the reader to
\cite{SCHSHA, SHA}.

\begin{remark}
While the results for the quintic were given in the case $5\nmid(p-1)$,
in \cite{CORV1} the case $5\mid(p-1)$ in which additional technical difficulties
arise is also covered.
\end{remark}

\begin{remark}
Since the calculation for the number of points is determined by character
formulas which are valid for any finite group, we can calculate the
number of points on (\ref{eq:quintic-def}) with coordinates in field
extensions $k_{r}=\F_{p^r}$ of $k=\F_{p}$ by producing multiplicative
and additive characters from these fields. If $q=p^r$, then the Teichm\"uller
character $T\colon k_r \to\Q_{p}$ is given by the expression
\[
T(x)=\lim_{n\to\infty}x^{q^{n}},
\]
while an additive character $\theta_{r}\colon k_r \to\Q_{p}$ is
given by composing $\theta$ with the (additive) trace map $\text{Tr}\colon k_r \to k$
given by $\text{Tr}(x)=x+x^{p}+\ldots+x^{p^{r-1}}$. 
\end{remark}

\section{Zeta Functions and Mirror Symmetry\label{sec:Zeta-Functions-and}}

Let $k=\F_{q}$ be a finite field with $q=p^{k}$ elements, $k_{r}$ an
extension of degree $r$, and $X$ a smooth variety set of dimension $d$
over $k$.  We will usually take $k=\F_{p}$. If we let $N_r(X)$ denote
the number of points of $\bar{X}:=X\times\bar{k}$ rational
over $k_{r}$, then the generating function
\[
Z(X,T)=\exp\left(\sum_{r=1}^{\infty}N_{r}(X) \frac{T^{r}}{r}\right)
\]
is called the \emph{zeta function} of $X$. From the
Weil Conjecture (since proved; see \cite{HARTSHORNE} for a broader overview) we know that
\[
Z(X,T)=\frac{P_{1}(T)P_{3}(T)\ldots P_{2d-1}(T)}{P_{0}(T)P_{2}(T)\ldots P_{2d}(T)},
\]
where $P_{0}(T)=1-T$, $P_{2d}(T)=1-q^{d}T$ and $P_{i}(T)\in1+T\Z[T]$
for each $1\le i\leq2d-1$. Furthermore, the degree of $P_{i}$ equals
the $i$-th Betti number $b_{i}$ of $X$ and
\begin{align*}
  P_{i}(T)
  &=\det\left( I - T\Frob^* \mid H^i(X) \right)
  &=\prod_{j=1}^{b_{i}}(1-\alpha_{ij}T),
\end{align*}
where $\alpha_{ij}$ are algebraic integers such that $|\alpha_{ij}|=q^{i/2}$, 
$H^i(X)$ is a suitable cohomology theory, for instance \'Etale cohomology,
and $\Frob^*$ is the map on cohomology induced from the Frobenius morphism
$\Frob\ colon \bar X \to \bar X$ given by $(x_i) \mapsto (x_i^q)$.

What can be said about the relationship between zeta functions of
a pair of mirror quintic threefolds $M$ and $W$ belonging, respectively,
to the quintic family $\cM$ defined by (\ref{eq:quintic-def}), and
its mirror $\cW$ outlined in Example \ref{exa:PFE-for-mirror-of-quintic}?
The Weil Conjecture and the Hodge diamond of $M$ imply that if
$Z(M,T)=N(T)/D(T)$, then we have $\deg N(T)=2h^{2,1}(M)+2=204$ and
$\deg D(T)=2h^{1,1}(M)+2=4$.  Since $h^{2,1}$ and $h^{1,1}$ are
exchanged under mirror symmetry, we might hope that there is some kind
of zeta function, which Candelas et al. call the ``quantum'' zeta
function in \cite{CORV2}, that satisfies $Z^{Q}(M,T)=Z^{Q}(W,T)^{-1}$.
This zeta function cannot be the usual zeta function, since that would
imply the impossible relation $N_{r}(M)=-N_{r}(W)$. However, numerical
calculations by Candelas, de la Ossa, and Rodriguez Villegas in
\cite{CORV2} show that if $\mid(p-1)$, then there is a relation
\[
Z(M,T)\equiv\frac{1}{Z(W,T)}\equiv(1-pT)^{100}(1-p^{2}T)^{100}\ \mod5^{2}.
\]
This congruence can be seen as coming from the fact that the zeta
functions of $M$ and $W$ share certain terms. In particular,
\begin{equation}\label{eq:zeta-Mpsi}
Z(M,T)=\frac{R_{\varepsilon}(T,\psi) \prod_{\vec v} R_{\vec v}(T, \psi)}{(1-T)(1-pT)(1-p^{2}T)(1-p^{3}T)}
\end{equation}
and
\begin{equation}\label{eq:zeta_Wpsi}
Z(W,T)=\frac{R_{\varepsilon}(T,\psi)}{(1-T)(1-pT)^{101}(1-p^{2}T)^{101}(1-p^{3}T)},
\end{equation}
where $R_\varepsilon(T,\psi)$ is of degree $4$, and each
$R_{\vec v}(T,\psi)$ comes from a period of $M_\psi$ as described
below. The relationship between the zeta function of a family of
manifolds and the solutions of the Picard-Fuchs equation was first
observed in greater generality by Katz in \cite{KATZ}. However,
because of its computational nature, we will first illustrate the
numerical calculation of \cite{CORV2}, and then proceed to outline
more conceptual explanations due to Kadir and Yui \cite{KYUIMM},
Kloosterman \cite{KLOOSTERMAN}, and Goutet
\cite{GOUTET1,GOUTET2,GOUTET3}.

\subsection{Computational Observations}

We have seen in Section~\ref{sec:Differentials-on-Hypersurfaces} that
periods of a hypersurface $X\subset\P^{n}$ given by $\{Q(x)=0\}$ are
determined by monomials
$x^{\vec{v}}=x_{1}^{v_{1}}\ldots x_{n+1}^{v_{n}}$ modulo those in the
Jacobian ideal $J(Q)$ of $Q$, since we can write every period in terms
of
\begin{equation}
\varpi_{v}:=\int_{\Gamma}\frac{x^{\vec{v}}\Omega}{Q^{k(\vec{v})}},\label{eq:period-from-monomial-xv}
\end{equation}
where $\Gamma$ is a cycle on $\P\setminus V$ described earlier,
and $k(\vec v)$ is determined by $k(\vec v)\deg\, Q=(v_{1}+\ldots+v_{5})+(n+1)$.
Using this description and Griffiths's formula (\ref{eq:ReductionFormula-1})
we can find relations amongst the periods, and in fact also Picard-Fuchs
equations, in a diagrammatic way. In the case of the quintic $X_{\psi}$
given in (\ref{eq:quintic-def}), choose some $i\in\{1,\ldots,n\}$ and
set $A_{i}=\frac{1}{5}x^{\vec v+ \vec{e_{i}}}$ as well as $A_{j}=0$ for $j\neq i$,
where $\vec{e_{i}}$ is the standard basis vector of $\Z^{n}$ with $1$
in the $i$-th slot and zeros elsewhere. If $\varepsilon=(1,\ldots,1)$,
then Griffiths's formula gives
\[
\frac{x^{\vec v}(x^{5\vec{e_{i}}}-\psi x^{\varepsilon})\Omega}{Q^{k(\vec v)+1}}=\frac{1}{5k(\vec v)}\frac{(v_{i}+1)x^{\vec v}\Omega}{Q^{k(\vec v)}}
\]
up to an exact form. If we use the shorthand $\vec v=(v_{1},v_{2},v_{3},v_{4},v_{5})$
for the period (\ref{eq:period-from-monomial-xv}) determined by the
monomial $x^{\vec v}$, then integrating this expression is a relation
between the three periods $\vec v$, $\vec v+5\vec{e_{i}}$ and $\vec v+\varepsilon$ which
we can encode in the diagram
\[
\begin{array}[t]{ccc}
\vec v & \rightarrow & \vec v+\varepsilon\\
\downarrow D_{i}\\
\vec v+5 \vec{e_{i}}
\end{array}
\]
where $D_{i}=\frac{\partial}{\partial x_{i}}\circ x_{i}$ denotes
the operator which gave rise the this relation. To get a differential
equation with respect to $\psi$ out of such relations we can use
\begin{equation}
\frac{\D }{\D \psi}\frac{x^{\vec v}\Omega}{Q^{k(\vec v)}}=\frac{-5k(\vec v)x^{\vec v+\varepsilon}}{Q^{k(\vec v)+1}},\label{eq:derivative-of-v-wrt-psi}
\end{equation}
which allows us to exchange $\vec v+\varepsilon$ for a derivative of $\vec v$.
\begin{example}
\label{exa:combinatorial-griffiths-dwork-example}Simply because the
diagrams are more manageable, we will illustrate this method on the
Fermat family of elliptic curves. Following \cite{CORV1}, we will also
change the form of the period integral encoded by the vector $\vec v$
to 
\[
\vec v=\frac{1}{(2\pi \i)^{3}}\int_{\Gamma}\frac{x^{\vec v}}{F_{\psi}^{k(\vec v)+1}}\, \D x_{1}\D x_{2}\D x_{3},
\]
where $Z_{\psi}\colon\{F_{\psi}=0\}$ defines an element of the family,
and $\Gamma$ is now a product of tubes around the loci $\partial Q/\partial x_{i}=0$.
Define $E:=x^{\varepsilon}=x_{1}x_{2}x_{3}$ and for $n\geq1$ let
\[
E_{n}:=\frac{E^{n-1}}{Q^{n}}\quad\text{and}\quad I_{n}:=\frac{1}{(2\pi \i)^{3}}\int_{\Gamma}E_{n}\ \D x_{1}\D x_{2}\D x_{3}.
\]
Applying the procedure above to $(0,0,0)$, $(1,1,1)$, and $(2,2,2)$
we have the diagram 
\[
\begin{array}{ccccccc}
 &  &  &  & (2,2,-1) & \rightarrow & (3,3,0)\\
 &  &  &  & \downarrow D_{3}\\
(0,0,0) & \rightarrow & (1,1,1) & \rightarrow & (2,2,2)\\
\downarrow D_{1} &  & \downarrow D_{1}\\
(3,0,0) & \rightarrow & (4,1,1)\\
\downarrow D_{2}\\
(3,3,0)
\end{array}
\]
in which the upper right dependence corresponds to 
\begin{equation}\label{eq:upper_right}
  D_{3}\left(\frac{x_{1}^{2}x_{2}^{2}}{x_{3}Q^{2}}\right)=
  \frac{\partial}{\partial x_{3}}\left(\frac{x_{1}^{2}x_{2}^{2}}{Q^{2}}\right)
  =  \frac{6\psi x_{1}^{3}x_{2}^{3}}{Q^{3}}-\frac{6E^{2}}{Q^{3}}.
\end{equation}
The relations coming from differentiation are 
\begin{equation}
I_{n+1}=\frac{1}{3n}\frac{\D }{\D \psi}I_{n}\quad\text{and}\quad I_{n+1}=\frac{1}{3^{n}n!}\frac{\D ^{n}}{\D \psi^{n}}I_{1}\,,\label{eq:gauss_manin_of_I1}
\end{equation}
which we use in the dependence diagram above. We start computing the
actual relations starting from the bottom of the diagram, replacing
terms until we have a relation between only the periods corresponding
to $(0,0,0),\ (1,1,1)$ and $(2,2,2)$. We get rid of the $(3,3,0)$
period because it ``loops around'' the diagram by equation (\ref{eq:upper_right}).
The end result is
\end{example}
\begin{align*}
E_{1}+3\psi E_{2}+6\left(\frac{\psi^{2}-1}{\psi}\right)E_{3} &=\frac{\partial}{\partial x_{1}}\left(\frac{x_{1}^{2}x_{2}x_{3}}{Q^{2}}+\frac{x_{1}}{Q}\right) \\ 
& \qquad +\frac{\partial}{\partial x_{2}}\left(\frac{x_{2}x_{1}^{3}}{Q^{2}}\right)+\frac{\partial}{\partial x_{3}}\left(\frac{x_{1}^{2}x_{2}^{2}}{\psi Q^{2}}\right),
\end{align*}
which results in the Picard-Fuchs equation 
\[
\left[3+3\psi\frac{\partial}{\partial\psi}+\left(\frac{\psi^{2}-1}{\psi}\right)\frac{\partial^{2}}{\partial\psi^{2}}\right]f(\psi)=0
\]
satisfied by the period
\[
I_{1}=\int_{\Gamma}\frac{\D x_{1}\D x_{2}\D x_{3}}{F_{\psi}}.
\]
Candelas et al. use such diagrams to find
the Picard-Fuchs equations for \emph{all }$204$ periods of the quintic
family (\ref{eq:quintic-def}). Note that $x^{\varepsilon}$ is invariant
under the diagonal symmetry group $\mathcal{G}$ of $X_{\psi}$ defined
in Example \ref{exa:PFE-for-mirror-of-quintic}. By equation (\ref{eq:derivative-of-v-wrt-psi})
this means that the periods $\varpi_{\vec v}$ and $\varpi_{\vec v+\varepsilon}$
correspond to the same representation of the group $\mathcal{G}$. Moreover,
the periods can be classified according to the transformation of $x^{\vec v}$
under the group into the sets
\begin{gather*}
\{1,x^{\varepsilon},x^{2\varepsilon},x^{3\varepsilon}\},\ \{x_{1}^{4}x_{2},x_{1}^{4}x_{2}x^{\varepsilon}\},\ \{x_{1}^{3}x_{2}^{2},x_{1}^{3}x_{2}^{2}x^{\varepsilon}\},\\
\{x_{1}^{2}x_{2}x_{3},x_{1}^{2}x_{2}x_{3}\},\ \{x_{1}^{2}x_{2}^{2}x_{3},x_{1}^{2}x_{2}^{2}x_{3}x^{\varepsilon}\},\text{ and }\{x_{1}^{4}x_{2}^{3}x_{3}^{2}x_{4}\},
\end{gather*}
given up to permutation of the variables.  In \cite{CORV1}, the
diagrammatic method is applied to each group of monomials. Choose a
representative $\vec v$ of each one, and denote the corresponding
Picard-Fuchs equation, which turns out to be hypergeometric in each
case, by $\cL_{\vec v}$. In the $p$-adic setting the hypergeometric
expressions for these allow, by comparison of coefficients, to rewrite
the number of points on the quintic (\ref{eq:quintic-def}) in terms of
all the periods as
\[
N(X_\psi)=p^{4}+\sum_{\vec v}\gamma_{\vec v}\sum_{m=0}^{p-2}\beta_{_{\vec v,m}}\Teich^{m}(\lambda),
\]
where the outer sum is over representative monomials in the sets above,
$\gamma_{\vec v}$ accounts for the number of permutations in each group,
and $\beta_{\vec v,m}$ is a ratio of Gauss sums
\[
\beta_{\vec v,m}=p^{4}\frac{G_{5m}}{\prod_{i=1}^{5}G_{m+k v_{i}}}.
\]
A consequence is an expression for the number of points
with coordinates in $k_r$ that decomposes as 
\begin{equation}\label{eq:quintic-nr-pts-decomp-into-periods}
N_{r}(X_\psi)= N_{\varepsilon,r}(X_\psi) + \sum_{\vec v}N_{\vec v,r}(X_\psi),
\end{equation}
so that $R_{\vec v}(T,\psi)$ arises as
$\sum_{r > 0} N_{\vec v,r} \frac{T^r}{r}$.  At the $\psi = 0$ (or
Fermat) point of the moduli space, equation
(\ref{eq:quintic-nr-pts-decomp-into-periods}) can equivalently be
given in terms of Fermat motives. This is a consequence of the
Kadir-Yui monomial-motive correspondence, which is a one-to-one
correspondence between the monomial classes given above and explicitly
realized Fermat motives. For more details and applications to mirror
symmetry, see \cite{KYUIMM}. The zeta function (\ref{eq:zeta_Wpsi})
can also be found by considering monomial classes, by understanding
the mirror $\cW$ torically and using \emph{Cox variables} instead of
$x^{\vec v}$. For more details, we refer the reader to \cite{CORV2}.

How can we interpret $\prod_{\vec v} R_{\vec v}(T,\psi)$ appearing in
(\ref{eq:zeta-Mpsi})?  Candelas et al. numerically observed for small
primes, and conjectured for all primes, that this product can be
written as
\[
R_A(qT,\psi)^{10} R_B(qT,\psi)^{15},
\]
where $R_A(T,\psi)$ and $R_B(T,\psi)$ arise as numerators of the
zeta functions of affine curves
\[
A\colon y^5 = x^2 (1-x)^3(x-\psi^5)^2 
\quad \text{and} \quad
B\colon y^5 = x^2(1-x)^4(x-\psi^5),
\]
respectively. This claim was proven by Goutet in \cite{GOUTET1} via
Gauss sum techniques. In fact, he has proven similar results more
generally.  An immediate generalization of (\ref{eq:quintic-def}) is the
Dwork family of hypersurfaces
\begin{equation}\label{eq:dwork-family}
X_\psi \colon \{ x_1^n + \ldots + x_n^n - n\psi x_1 \ldots x_n = 0 \} \subseteq \P_k^{n-1},
\end{equation}
where $\psi \in k$ and we only consider nonsingular $X_\psi$.
Arithmetic of this family and its mirror was considered by Wan
in \cite{WAN} and Haessig in \cite{HAESSIG}.  The mirror family is
constructed in two stages, analogously to the quintic case. First we
form the quotient $Y_\psi := X_\psi / G$, where
\[
G = \left\{ (\xi_1, \ldots, \xi_n) \mid \xi_i \in k, \xi_i^n=1, \xi_1 \ldots \xi_n = 1 \right\}
\]
is the group of diagonal symmetries of $X_\psi$. Wan calls $Y_\psi$
the \emph{singular mirror} of $X_\psi$. It can be explicitly realized
as the projective closure of the affine hypersurface
\[
g(x_{1},\ldots,x_{n-1})=x_{1}+\ldots+x_{n-1}+\frac{1}{x_{1}\ldots x_{n-1}}-n\psi=0,
\]
in the torus $(k^{*})^{n-1}$, which enables the use of Gauss sums to
count points. The mirror family $\{ W_\psi \}$ is obtained by
resolving the singularities of $\{ Y_\psi \}$. Picking a manifold from
each family will produce a mirror pair, and if the two parameter
values $\psi$ are equal, then $\{X_{\psi},W_{\psi}\}$ is called a
\emph{strong mirror pair}.  Now, reciprocal zeros $\beta_{i}$ and
poles $\gamma_{i}$ of the zeta function
$Z(X,T)=\prod_{i}(1-\beta_{i}T)/\prod_{j}(1-\gamma_{j}T)$) of some
smooth variety $X$ determine the number of points over various
extensions of $k$, since we have that
\[
\sum_{r=1}^{\infty}N_{r}(X) T^{r}=t\frac{\D \log(\zeta T)}{\D T}=\sum_{j}\frac{\gamma_{j}T}{1-\gamma_{j}T}-\sum_{i}\frac{\beta_{i}T}{1-\beta_{i}T},
\]
which implies
\[
N_{r}(X)=\sum_{j}\gamma_{j}^{r}-\sum_{i}\beta_{i}^{r}.
\]
Furthermore, if we define the slope of $\alpha\in\Q$ as 
\[
s(\alpha)=\text{ord}_{p}(\alpha)\,
\]
where $\text{ord}_{p}$ denotes the $p$-adic order of $\alpha$,
then $\beta_{i},\gamma_{j}$ as defined above satisfy 
\[
0\leq s(\beta_{i}),s(\gamma_{j})\leq2d
\]
and are rational numbers in the range $[0,\dim X]$. We now select
a part of the zeta function of $X$ 
\[
Z_{[0,1)}(X,t)=\prod_{\alpha_{i}\in\{\beta_{i},\gamma_{j}\},0\leq s(\alpha_{i})<1}(1-\alpha_{i}t)^{\pm1}.
\]
A character formula calculation gives the following theorem, which is
the main result of \cite{WAN}.
\begin{theorem}
  For a strong mirror pair $(X_\psi, W_\psi)$ and $r \in \Z_{>0}$ we have
  \[
  N_{r}(X_{\psi})\equiv N_{r}(Y_\psi) \equiv N_{r}(W_{\psi})\,\mod\, q^{r},
  \]
  or equivalently
  \[
  Z_{[0,1)}(X_{\psi},T)= Z_{[0,1)}(Y_\psi, T) = Z_{[0,1)}(W_{\psi},T).
  \]
\end{theorem}
In fact, if $q \equiv 1 \text{ mod } n$ or if $n$ is prime, it is
shown in \cite{WAN} and \cite{HAESSIG}, respectively, that
\[
Z(X_\psi,T) = \frac{ \left( Q(T,\psi) R(q^s T^s)\right)^{(-1)^{n-1}}}{(1-T) (1 - qT) \ldots (1-q^{n-2}T)},
\]
where $s$ is the order of $q$ in $(\Z/n\Z)^n$, as well as
\[
Z(Y_\psi, T) = \frac{ Q(T,\psi)^{(-1)^{n-1}}}{(1-T) (1 - qT) \ldots (1-q^{n-2}T)}.
\]
A natural question to ask is whether, analogously to the quintic case,
the polynomial $R(T,\psi)$ can be shown to contain terms appearing in
zeta functions of other varieties. As remarked by Wan in \cite{WAN},
in addition to the $n = 5$ case, this question was answered
affirmatively in the cases $n = 3$ and $n = 4$ by Dwork.  Relying on a
result of Haessig \cite{HAESSIG} and Gauss sum calculations, Goutet
\cite{GOUTET2} has found explicit varieties whose zeta functions
have terms appearing in $R(T,\psi)$. Specifically, if we define
$N_R(q^r)$ by
$R(T,\psi) = \exp\left( \sum_{r > 0} N_R(q^r) \frac{T^r}{r} \right)$,
Goutet proves the following.
\begin{theorem}\label{thm:goutet-hypergeometric-factors}
  Let $n \geq 5$ be a prime congruent to $1$ modulo $n$. Then,
  \[
  N_R(q^r) = q^{\frac{n-5}{2}} N_1(q^r) + q^{\frac{n-7}{2}}N_3(q^r) + \ldots + N_{n-4}(q^r),
  \]
  where each $N_i(q^r)$ is equal to the sum of counts of points of
  certain varieties of hypergeometric type.
\end{theorem}

\subsection{Cohomological Interpretation}

While Gauss sum techniques allow us to test and prove conjectures
about zeta functions of mirror manifolds, they do not provide a
conceptual understanding of what is happening.  We have already
mentioned the Kadir-Yui monomial-motive correspondence \cite{KYUIMM}
which begins to provide a theoretical explanation. Kloosterman
\cite{KLOOSTERMAN} extends this result to a neighborhood of the Fermat
fiber, but also considers more general families. In particular, let
$k = \F_q$ be a finite field and consider the family consisting of
hypersurfaces
\begin{equation}\label{eq:fermat-monomial-deformation}
X_{\bar \lambda} \colon \left\{ F_{\bar \lambda} = \sum_{i=0}^n x_i^{d_i} + \bar \lambda \prod_i x_i^{a_i} = 0 \right \}
\end{equation}
in weighted projective space $\P(w):=\P_k(w_0,\ldots,w_n)$, where
$w_i d_i = d$, $a_i \geq 0$, $\gcd(q,d)=1$, and $\sum w_i a_i = d$.
We will also only work with nonsingular fibers in what follows.
If $U_{\bar \lambda} = \P(w) \setminus X_{\bar \lambda}$
denotes the complement of a generic member of this family, then
\[ 
Z(X_{\bar \lambda},T) Z(U_{\bar \lambda}, T) = Z(\P(w),T),
\]
and we can work with $U_{\bar \lambda}$ instead of $X_{\bar \lambda}$
for the purposes of determining the zeta function of
$X_{\bar \lambda}$. One reason for doing so is that there is a
$p$-adic cohomology theory resembling de Rham cohomology called
\emph{Monsky-Washnitzer cohomology} that is well understood on
hypersurface complements. To work with Monsky-Washnitzer cohomology,
we need to lift $X_{\bar \lambda}$ to a $p$-adic context.  Let
$\lambda$ be the Teich\"uller lift of $\bar \lambda$ to the fraction
field $\Q_q$ of the ring of Witt vectors over $k = \F_q$ (which equals
$\Z_p$ if $q = p$). We can then consider $F_\lambda$ to have
coefficients in $\Q_q$, and work with $X_\lambda$ and $U_\lambda$
defined in the obvious way over $\Q_q$.  Cohomology classes of
Monsky-Washnitzer cohomology $H_{MW}^*(U_\lambda,\Q_q)$ are given by
differential forms with $\Q_q$ coefficients, and these groups possess
an action of Frobenius. It turns out that cohomology is zero except
in degree $n$ and degree $0$, where it is one-dimensional with 
trivial action of Frobenius. From this it can be shown that
\[
Z(U_{\bar \lambda}, T) = \frac{\left( \det \left(  I - q^n(\Frob_q^*)^{-1}T \mid H_{MW}^n(U_\lambda, \Q_q)   \right) \right)^{(-1)^{n+1}}}{(1-q^nT)}.
\] 
By a result of Katz \cite{KATZ}, $(\Frob_q^*)^{-1}$ can be given by
$A(\lambda)^{-1} \Frob_{q,0}^* A(\lambda^q)$ extended via $p$-adic
analytic continuation to a small disc around $\lambda = 0$, where
$\Frob_{q,0}^*$ is the action of Frobenius on the $\lambda = 0$ fiber,
and $A(\lambda)$ is a solution of the Picard-Fuchs equation associated
with the family $X_\lambda$.  Therefore, to determine the zeta
function of $X_{\bar \lambda}$, we need to understand the action of
Frobenius on the Fermat fiber, and to compute the Picard-Fuchs
equation of the deformed family.  Finding the latter and showing it is
hypergeometric is one of the main results of \cite{KLOOSTERMAN}.
Additionally, Kloosterman shows that there is a factorization of the
zeta function along the lines of the Kadir-Yui monomial-motive
correspondence \cite{KYUIMM}. These ideas were also exploited 
to calculate zeta functions of certain K3 surfaces in \cite{GOKLYU}.

An alternative theoretical approach, in terms of \'Etale
cohomology, is given by Goutet in \cite{GOUTET3}. For a
nonsingular element $\bar X_\psi = X_\psi \times_k \bar k$ of this
family considered over $\bar k$, it can be shown that
$H_{\textnormal{et}}(\bar X_\psi, \Q_\ell)$ is zero for $i > 2n-4$ and
$i < 0$, as well as for odd $i \neq n-2$. For the remaining even
$i \neq n-2$ these groups are $1$-dimensional. The most interesting
part of cohomology is thus the primitive part of
$H_{\textnormal{et}}^{n-2}(\bar X_\psi, \Q_\ell)$, since it can be
shown that the action of Frobenius is multiplication by $q^{(n-2)/2}$ on the
non-primitive part of
$H_{\textnormal{et}}^{n-2}(\bar X_\psi, \Q_\ell)$, and 
multiplication by $q^i$ on each
$H_{\textnormal{et}}^{2i} (\bar X_\psi, \Q_\ell)$. It follows that
\[
Z(X_\psi,T) = \frac{\left( \det \left( I - T \Frob^* \mid H_{\textnormal{et}}^{n-2}(\bar X_\psi, \Q_\ell)^{\textnormal{prim}} \right) \right)^{(-1)^{n-1}}}{(1-T)(1-qT)\ldots(1-q^{n-2}T)}.
\]
Goutet shows that $H_{\textnormal{et}}(\bar X_\psi, \Q_\ell)^{\textnormal{prim}}$ 
decomposes into a direct sum of linear subspaces which correspond to 
equivalence classes of irreducible representations of the group of 
automorphisms of $X_\psi$ acting on cohomology. Frobenius stabilizes
each of these subspaces, and the zeta function inherits a factor from each
summand. The resulting factorization is finer than the one
given in \cite{KLOOSTERMAN}, and Goutet relates this factorization
to the one resulting from Theorem~\ref{thm:goutet-hypergeometric-factors}
in a recent preprint \cite{GOUTET4}. An interesting question is whether
these factors can be explained geometrically in the context of mirror symmetry.

\section*{Acknowledgments}

The author would like to thank the anonymous referee for helpful
remarks that resulted in large improvements to this document.  Thanks
is also due to Professor Noriko Yui for helpful suggestions and
tireless encouragement during the preparation of this manuscript.  The
author's work is supported by the Natural Sciences and Engineering
Research Council (NSERC) of Canada through the Discovery Grant of
Noriko Yui. The author held a visiting position at the Fields
Institute during the preparation of these notes, and would like to
thank this institution for its hospitality.

\bibliography{notes}{}

\begin{thebibliography}{10}
\providecommand{\url}[1]{{#1}}
\providecommand{\urlprefix}{URL }
\expandafter\ifx\csname urlstyle\endcsname\relax
  \providecommand{\doi}[1]{DOI~\discretionary{}{}{}#1}\else
  \providecommand{\doi}{DOI~\discretionary{}{}{}\begingroup
  \urlstyle{rm}\Url}\fi

\bibitem{CANMS}
{Candelas}, P., {de la Ossa}, X., {Green}, P.S., {Parkes}, L.: {A pair of
  Calabi-Yau manifolds as an exactly soluble superconformal theory}.
\newblock Nuclear Physics B \textbf{359}(1), 21--74 (1991)

\bibitem{CORV1}
{Candelas}, P., {de la Ossa}, X., {Rodriguez-Villegas}, F.: {Calabi-Yau
  manifolds over finite fields, I}.
\newblock arXiv preprint hep-th/0012233  (2000)

\bibitem{CORV2}
{Candelas}, P., {de la Ossa}, X., {Rodriguez-Villegas}, F.: {Calabi-Yau
  manifolds over finite fields, II}.
\newblock Fields Institute Communications \textbf{38}, 121--157 (2003)

\bibitem{COKA}
{Cox}, D.A., {Katz}, S.: {Mirror Symmetry and Algebraic Geometry}.
\newblock American Mathematical Soc. (1999)

\bibitem{DOLG}
{Dolgachev}, I.: {Weighted Projective Varieties}.
\newblock In: {Group Actions and Vector Fields}, pp. 34--71. Springer (1982)

\bibitem{DWO}
{{Dwork}, Bernard}: {Generalized Hypergeometric Functions}.
\newblock {Oxford Mathematical Monographs}. Oxford University Press (1990)

\bibitem{GAHRS}
{G{\"a}hrs}, S.: {Picard--Fuchs Equations of Special One-Parameter Families of
  Invertible Polynomials}.
\newblock Springer (2013)

\bibitem{GOKLYU}
{Goto}, Y., {Kloosterman}, R., {Yui}, N.: {Zeta-functions of certain K3-fibered
  Calabi--Yau threefolds}.
\newblock International Journal of Mathematics \textbf{22}(01), 67--129 (2011)

\bibitem{GOUTET2}
{Goutet}, P.: {An explicit factorisation of the zeta functions of Dwork
  hypersurfaces}.
\newblock Acta Arithmetica \textbf{144}(3), 241--261 (2010)

\bibitem{GOUTET1}
{Goutet}, P.: On the zeta function of a family of quintics.
\newblock Journal of Number Theory \textbf{130}(3), 478--492 (2010)

\bibitem{GOUTET3}
{Goutet}, P.: {Isotypic decomposition of the cohomology and factorization of
  the zeta functions of Dwork hypersurfaces}.
\newblock Finite Fields and Their Applications \textbf{17}(2), 113--147 (2011)

\bibitem{GOUTET4}
{Goutet}, P.: {Link between two factorizations of the zeta functions of Dwork
  hypersurfaces}.
\newblock preprint  (2014)

\bibitem{GRHA}
{Griffiths}, P., {Harris}, J.: {Principles of Algebraic Geometry}, vol.~52.
\newblock John Wiley \& Sons (2011)

\bibitem{GRIF1}
{Griffiths}, P.A.: {On the periods of certain rational integrals: I}.
\newblock Annals of Mathematics pp. 460--495 (1969)

\bibitem{GRIF2}
{Griffiths}, P.A.: {On the periods of certain rational integrals: II}.
\newblock Annals of Mathematics pp. 496--541 (1969)

\bibitem{HAESSIG}
{Haessig}, C.D.: Equalities, congruences, and quotients of zeta functions in
  arithmetic mirror symmetry.
\newblock In: {Mirror Symmetry V}, \emph{{AMS/IP Studies in Advanced
  Mathematics}}, vol.~38, pp. 159--184. {American Mathematical Society} (2007)

\bibitem{HARTSHORNE}
{Hartshorne}, R.: {Algebraic Geometry}.
\newblock Springer (1977)

\bibitem{KATZ}
{Katz}, N.M.: On the differential equations satisfied by period matrices.
\newblock Publications Math{\'e}matiques de l'IH{\'E}S \textbf{35}(1), 71--106
  (1968)

\bibitem{KLOOSTERMAN}
{Kloosterman}, R.: {The zeta function of monomial deformations of Fermat
  hypersurfaces}.
\newblock Algebra \& Number Theory \textbf{1}(4), 421--450 (2007)

\bibitem{KOBLITZ}
{Koblitz}, N.: {$p$-adic Numbers, $p$-adic Analysis, and Zeta-Functions}, 2nd
  edn.
\newblock Springer (1984)

\bibitem{LLY}
{Lian}, B., {Liu}, K., {Yau}, S.T.: {Mirror principle I}.
\newblock arXiv preprint alg-geom/9712011  (1997)

\bibitem{MOR}
{Morrison}, D.R.: {Picard-Fuchs equations and mirror maps for hypersurfaces}.
\newblock arXiv preprint alg-geom/9202026  (1992)

\bibitem{SCHSHA}
{Schwarz}, A., {Shapiro}, I.: {Twisted de Rham cohomology, homological
  definition of the integral and ``Physics over a ring''}.
\newblock Nuclear Physics B \textbf{809}(3), 547--560 (2009)

\bibitem{SHA}
{Shapiro}, I.: {Frobenius map for quintic threefolds}.
\newblock International Mathematics Research Notices \textbf{2}(13), 2519--2545
  (2009)

\bibitem{VOIS1}
{Voisin}, C.: {Hodge Theory and Complex Algebraic Geometry I}, vol.~1.
\newblock Cambridge University Press (2008)

\bibitem{WAN}
{Wan}, D.: {Mirror symmetry for zeta functions}.
\newblock AMS/IP Studies in Advanced Mathematics \textbf{38}, 159--184 (2006)

\bibitem{KYUIMM}
{Yui}, N., {Kadir}, S.: {Motives and mirror symmetry for Calabi-Yau orbifolds}.
\newblock In: {Modular Forms and String Duality}, \emph{Fields Institute
  Communications}, vol.~54, pp. 3--46. American Mathematical Society (2008)

\end{thebibliography}
\bibliographystyle{spmpsci.bst}
\end{document}